\documentclass[a4paper,oneside,10pt]{article}%
\usepackage{amsmath}
\usepackage{amsfonts}
\usepackage{amssymb}
\usepackage[ruled,vlined]{algorithm2e} 
\usepackage{multirow}
\usepackage{graphicx}
\usepackage{epstopdf}
\usepackage{array}
\usepackage{booktabs}
\usepackage{caption}
\usepackage{subfigure}
\usepackage[square,numbers,sort&compress]{natbib}
\usepackage{makecell}
\usepackage{diagbox}%
\usepackage{mathrsfs}
\providecommand{\U}[1]{\protect\rule{.1in}{.1in}}
\allowdisplaybreaks[4]

\newtheorem{theorem}{Theorem}[section]

\newtheorem{lemma}[theorem]{Lemma}

\newtheorem{proposition}[theorem]{Proposition}

\newenvironment{proof}[1][Proof]{\noindent \textbf{#1.} }{\  \rule{0.5em}{0.5em}}
\newcounter{scheme}[section]
\newcounter{definition}[section]
\newcounter{remark}[section]
\renewcommand{\thescheme}{Scheme~\thesection.\arabic{scheme}}
\renewcommand{\thedefinition}{Definition~\thesection.\arabic{definition}}
\renewcommand{\theremark}{Remark~\thesection.\arabic{remark}}
\numberwithin{equation}{section}
\renewcommand \arraystretch{1.5}

\begin{document}
\title{A novel second order scheme with one step for forward backward stochastic differential equations}
\author{Qiang Han \thanks{School of Mathematical Science, Yangzhou University, Yangzhou 225002, China.
Email: hanqiang@yzu.edu.cn.}
\and Shihao Lan\thanks{School of Mathematical Science, Yangzhou University, Yangzhou 225002, China. Email: mx120230338@stu.yzu.edu.cn}
\and Quanxin Zhu\thanks{CHP-LCOCS, School of Mathematics and Statistics, Hunan Normal University, Changsha 410081, China. Email: zqx22@126.com
(Corresponding author).  This work was jointly supported by the National Natural Science Foundation of
China (62173139), and the Natural Science Foundation of Hunan Province, China (2023JJ30388).}}


\maketitle

\textbf{Abstract}.
In this paper, we present a novel explicit second order scheme with one step for solving the forward backward stochastic differential equations, with the Crank-Nicolson method as a specific instance within our proposed framework. We first present a rigorous stability result, followed by precise error estimates that confirm the proposed novel scheme achieves second-order convergence. The theoretical results for the proposed methods are supported by numerical experiments.



\addcontentsline{toc}{section}{\hspace*{1.8em}Abstract}

\section{Introduction}

Consider a complete probability space $(\Omega, \mathcal{F}, \mathbb{F}, \mathbb{P})$, where $\mathbb{F} = (\mathcal{F}_{t})_{0 \leq t \leq T}$ represents the natural filtration generated by a standard $q$-dimensional Brownian motion. Within this probability space, this paper does not prioritize the regularity of the solutions to the forward backward stochastic differential equations (FBSDEs) but rather focuses on probabilistic numerical methods for solving them:

\begin{equation}
\left\{
\begin{aligned}
X_{t}= & X_{0}+\int_{0}^{t}b(s,X_{s})ds+\int_{0}^{t}\sigma(s,X_{s}%
)dW_{s},\quad\qquad\qquad(SDE)\\
Y_{t}= & \Phi(X_{T})+\int_{t}^{T}f(s,Y_{s},Z_{s})ds-\int_{t}^{T}%
Z_{s}dW_{s},\qquad\qquad(BSDE)
\end{aligned}
\right.  \label{dFBSDE}%
\end{equation}
where $T > 0$ denotes a fixed terminal time; $X_{0} \in \mathbb{R}^q$ provides the initial condition for the SDE; $\Phi(X_{T}) \in \mathbb{R}^n$ specifies the terminal condition for the BSDE; $b: [0,T] \times \mathbb{R}^q \to \mathbb{R}^q$ is the drift coefficient; $\sigma: [0,T] \times \mathbb{R}^q \to \mathbb{R}^{q \times q}$ represents the diffusion matrix; $f: [0,T] \times \mathbb{R}^n \times \mathbb{R}^{n \times q} \to \mathbb{R}^n$ is the generator function; and $\Phi: \mathbb{R}^q \to \mathbb{R}^n$ is a real-valued function. To facilitate analysis, we assume that the functions $b$, $\sigma$, $f$, and $\Phi$ adhere to the following conditions:
\begin{description}
\item[(i)]
The functions $b$ and $\sigma \in C_b^1$, where $C_b^k$ denotes the space of continuous functions with uniformly bounded derivatives up to order $k$. In particular, we assume that:
\[
\sup\limits_{0\le t\le T}\{|b(t,0)|+|\sigma(t,0)|\}\leq L,
\]
with the non-negative constant $L$ representing the Lipschitz constants for all relevant functions.

\item[(ii)]
We assume that $\sigma$ satisfies:
$$
\sigma(t,x)\sigma^\top(t,x)\ge \frac{I_q}{L},  \forall (t,x)\in[0,T]\times \mathbb{R}^{q}.
$$

\item[(iii)]
$b$, $\sigma$, $f$, and $\Phi \in \mathcal{C}_L$, where $\mathcal{C}_L$ denotes the set of uniformly Lipschitz continuous functions with respect to the spatial variables. Additionally, we assume:
\[
\sup_{0 \leq t \leq T} \left\{|f(t, 0, 0)| + |\Phi(0)|\right\} \leq L.
\]
\end{description}

Under these conditions, the FBSDEs (\ref{dFBSDE}) is well-posed. Additionally, by taking the conditional expectation on both sides of the backward component, the integrands are shown to be continuous with respect to time.

The well-posedness of the FBSDEs (\ref{dFBSDE}), including both existence and uniqueness of the solutions, is established in \cite{PEPS90}. Moreover, the papers \cite{PS91,PEPS92} show that the solution $(Y_t, Z_t)$ to the FBSDEs (\ref{dFBSDE}) can be expressed as
\begin{equation}
Y_{t} = u(t,X_{t}),\qquad Z_{t} = \nabla u(t,X_{t})\sigma(t,X_{t})
\qquad\forall t\in[0,T],\label{2-2}%
\end{equation}
where $\nabla u$ denotes the gradient of $u(t,x)$ with respect to $x$; $u(t,x) \in C_b^{1,2}([0,T] \times \mathbb{R}^q)$ is the solution of the following nonlinear parabolic partial differential equation:
\begin{equation}
\mathcal{L}^{(0)}u(t,x)+f(u(t,x),\nabla u(t,x)\sigma(t,x))=0,
\label{2-3}%
\end{equation}
with the terminal condition $u(T, x) = \Phi(x)$; and $\mathcal{L}^{(0)}=\frac{\partial}{\partial t}+\sum\limits_{i=1}^{q}b_{i}(t,x)\frac{\partial
}{\partial x_{i}}+ \frac{1}{2}\sum\limits_{i,j=1}^{q}(\sigma\sigma^{\top})_{ij}(t,x)%
\frac{\partial^{2}}{\partial x_{i}\partial x_{j}}$.

The interest in FBSDEs has grown among researchers due to their extensive applications across disciplines, including fields such as mathematical finance \cite{CEZ23,EPQ97}, stochastic optimal control problems \cite{ZZZQ24,ZYZQ24}, risk measures \cite{KAWJXB92,XY16}, and more within the social and natural sciences.
 However, finding analytical solutions to FBSDEs is rarely feasible because the majority of these equations are complex. Thus, numerical algorithms have to be constructed to approximate their solutions. Consequently, significant work has been devoted to the numerical solutions of FBSDEs.
 Specifically, the Euler schemes, which have been shown to achieve a convergence order of $\frac{1}{2}$, are discussed in \cite{ZJ04, BBTN04, GELC07, PSXM11, HMJL21}. For higher orders of convergence, the multistep schemes, which generalize the Euler schemes by utilizing more previously known information, have been developed and analyzed in \cite{BCDR07, BCSM23, BCSJ12, CJF14, GELSJGVC20, GETP16b,  HQJS21b, HQJS21a, HQJS21c, HQJS22, TXXJ22}. However, since the multistep schemes require multiple previous solution points to calculate the next one, they necessitate several initial values, increasing computational complexity.

To address the limitations of the multistep schemes, which require multiple initial values, and to maintain higher orders of convergence, this paper introduces a novel one step second order numerical scheme that generalizes the Euler schemes. The proposed scheme, developed using a predictor-corrector approach, is entirely explicit for $Y$ and $Z$.
When solving FBSDEs, the explicit schemes compute the next time step directly from known values, bypassing the need to solve the nonlinear equations required by implicit methods, thereby enabling faster computations. Additionally, the explicit schemes avoid the complex matrix operations and iterative processes inherent in implicit methods, significantly reducing computational complexity, particularly in high-dimensional problems.
Moreover, a stability result will be established to provide precise error estimates and confirm that our method achieves second order convergence.
The key distinctions between this paper and the references \cite{ZWCLPS06, CDMK14} are that our one step second order scheme is fully explicit for both $Y$ and $Z$, whereas the schemes in \cite{ZWCLPS06} and \cite{CDMK14} are explicit for $Y$ while being implicit for $Z$. Furthermore, when $\alpha = 1$, our approach simplifies to the Crank-Nicolson method discussed in \cite{CDMK14, ZWCLPS06}, highlighting the Crank-Nicolson method as a particular case within our framework.

The structure of this paper is as follows: Section 2 outlines the explicit one step method used to solve the FBSDEs (\ref{dFBSDE}). In Section 3, we conduct a comprehensive stability analysis of the proposed method. Section 4 focuses on evaluating the local truncation errors for $Y$ and $Z$, and discusses the second order convergence of the numerical approach. In Section 5, numerical experiments are presented to verify the theoretical findings. Lastly, Section 6 concludes the paper.



\section{Novel one step discretization of the BSDE}

This section presents the time-discretization of the BSDE in (\ref{dFBSDE}) using a novel one-step scheme. To achieve this, a uniform discrete mesh $\pi = \{t_0, t_1, \cdots, t_N\}$ is defined over the time interval $[0, T]$, with the step size $h_t = \frac{T}{N}$, where $N \in \mathbb{N}^+$. The BSDE in (\ref{dFBSDE}) is thus expressed at the mesh points $t_i$ as follows
\begin{equation}
Y_{t_i}= Y_{t_{i+1}} + \int_{t_{i}}^{t_{i+1}} f_sds - \int_{t_{i}%
}^{t_{i+1}} Z_{s} d W_{s}, \label{FBSDE-t-i}%
\end{equation}
where $f_s = f(s,Y_{s},Z_{s})$.
Applying conditional expectations to both sides of (\ref{FBSDE-t-i}), we get
\begin{equation}%
Y_{t_i}  = \mathbb{E}_{t_i}^x[Y_{t_{i+1}}] + \int_{t_{i}}^{t_{i+1}} \mathbb{E}_{t_i}^x[f_{s}]ds,
\label{FBSDE-t-i-E}%
\end{equation}
where $\mathbb{E}_{t_i}^x[\cdot]=\mathbb{E}[\cdot|\mathcal{F}_{t_i},X_{t_i}=x]$.
As explained in Section 1, the integrand $\mathbb{E}_{t_i}^x[f_{s}]$ in the equation above depends deterministically on $s$. Thus, the one step method can be used to approximate it, namely
\begin{align}
 \int_{t_i}^{t_{i+1}}\mathbb{E}_{t_i}^x[f_{s}]ds =h_t \mathbb{E}_{t_i}^x[\frac{1}{2\alpha}f_{t_{i+1-\alpha }}+ (1-\frac{1}{2\alpha})f_{t_{i+1}} ] + R^i_{yf}, \label{FBSDE-err-y}
\end{align}
where $f_{t_{i+1-\alpha}}=f(t_{i+1-\alpha},Y_{t_{i+1-\alpha}},Z_{t_{i+1-\alpha}})$, $\alpha\in(0,1]$ and
$ R^i_{yf} = \int_{t_i}^{t_{i+1}}\mathbb{E}_{t_i}^x[f_s]ds - h_t \mathbb{E}_{t_i}^x[\frac{1}{2\alpha}f_{t_{i+1-\alpha }}+ (1-\frac{1}{2\alpha})f_{t_{i+1}}]$.
Inserting (\ref{FBSDE-err-y}) into (\ref{FBSDE-t-i-E}), we obtain
\begin{align}%
Y_{t_i}  = \mathbb{E}_{t_i}^x[Y_{t_{i+1}} + \frac{h_t}{2\alpha}f_{t_{i+1-\alpha }}+ h_t(1-\frac{1}{2\alpha})f_{t_{i+1}}] + R^i_{yf}.
\label{FBSDE-err-y-1}%
\end{align}

The values of $(Y_{t_{i+1-\alpha}},Z_{t_{i+1-\alpha}})$ are at non-grid points and unknown for $i=0,1,\cdots,N-1$.
Calculating $Y_{t_i}$ with an explicit numerical scheme requires approximating $Y_{t_{i+1-\alpha}}$ and $Z_{t_{i+1-\alpha}}$ in $f_{t_{i+1-\alpha}}$.
Firstly, we approximate the value of $Y_{t_{i+1-\alpha}}$ in $f_{t_{i+1-\alpha}}$ and then the computation expression with respect to $Z_{t_{i+1-\alpha}}$ will be given soon.
We approximate the value of $Y_{t_{i+1-\alpha}}$  by the explicit Euler scheme, namely
\begin{equation}%
Y_{t_{i+1-\alpha}}  = \mathbb{E}_{t_{i+1-\alpha}}^x[Y_{t_{i+1}}+ \alpha h_tf_{t_{i+1}}]+\widetilde{R}^i_{yf},
\label{FBSDE-err-y-2}%
\end{equation}
where $\mathbb{E}_{t_{i+1-\alpha}}^x[\cdot]=\mathbb{E}[\cdot|\mathcal{F}_{t_{i+1-\alpha}},X_{t_{i+1-\alpha}}=x]$;
$f_{t_i}=f(t_i,Y_{t_i},Z_{t_i})$ for $i=0,1,\cdots,N-1$;
$\widetilde{R}^i_{yf} =  \int_{t_{i+1-\alpha}}^{t_{i+1}} \mathbb{E}_{t_{i+1-\alpha}}^x[f_{s}]
ds-\alpha h_t\mathbb{E}_{t_{i+1-\alpha}}^x[f_{t_{i+1}} ]$.
Thus, the computation expression with respect to $Y_{t_{i+1-\alpha}}$ is,
\begin{align}
Y^\pi_{i+1-\alpha}
=&\mathbb{E}_{t_{i+1-\alpha}}^x\left[Y^\pi_{i+1}+\alpha h_tf^\pi_{i+1}\right],\label{E-GBSDE-t-i-pi}
\end{align}
where $f_{i}^{\pi}=f(t_{i},Y_{i}^{\pi},Z_{i}^{\pi})$.

Then, we derive the expression for $Z$. By multiplying (\ref{FBSDE-t-i}) with $\Delta W_{i,i+1}^\top:= (W_{t_{i+1}}-W_{t_i})^\top$ and applying the conditional expectation, we obtain
\begin{equation}%
0= \mathbb{E}_{t_i}^x[Y_{t_{i+1}}\Delta W_{i,i+1}^{\top}]+\int_{t_{i}
}^{t_{i+1}}\mathbb{E}_{t_i}^x[f_s\Delta W_{t_i,s}^{\top}]ds-%
\int_{t_{i}}^{t_{i+1}}\mathbb{E}_{t_i}^x[Z_{s}]ds,\label{FBSDE-err-z}
\end{equation}
where $\Delta W_{t_i,s}= W_{s}-W_{t_i}.$
Similarly, we approximate the integral terms $\mathbb{E}_{t_i}^x[f_s\Delta W_{t_i,s}^{\top}]$ and $\mathbb{E}_{t_i}^x[Z_{s}]$ on the right-hand side of (\ref{FBSDE-err-z}) using the same method as for calculating $Y_{t_i}$, namely
\begin{equation}
\int_{t_{i}}^{t_{i+1}}\mathbb{E}_{t_i}^x[f(s,Y_{s},Z_{s})\Delta W_{t_i,s}^{\top}]ds
= h_t\mathbb{E}_{t_i}^x[\frac{1}{2\alpha}f_{t_{i+1-\alpha }}\Delta W_{i,i+1-\alpha}^{\top}+ (1-\frac{1}{2\alpha})f_{t_{i+1}}\Delta W_{i,i+1}^{\top}]+ R^i_{z,1},\label{FBSDE-err-z-1}
\end{equation}
\begin{equation}
\int_{t_{i}}^{t_{i+1}}\mathbb{E}_{t_i}^x[Z_{s}]ds
 = \frac{h_t}{2}\mathbb{E}_{t_i}^x[Z_{t_i}+ Z_{t_{i+1}}]+ R_{z,2}^i,\label{FBSDE-err-z-2}
\end{equation}
where $\Delta W_{i,i+1-\alpha}= W_{t_{i+1-\alpha}}-W_{t_i}$,
\begin{equation}
\begin{array}
[c]{rl}%
&R_{z,1}^i= \int_{t_{i}}^{t_{i+1}}\mathbb{E}_{t_i}^x[f(s,Y_{s},Z_{s})\Delta W_{t_i,s}^{\top}]ds- h_t\mathbb{E}_{t_i}^x[\frac{1}{2\alpha}f_{t_{i+1-\alpha }}\Delta W_{i,i+1-\alpha}^{\top}+ (1-\frac{1}{2\alpha})f_{t_{i+1}}\Delta W_{i,i+1}^{\top}],\\
&R_{z,2}^i=\int_{t_{i}}^{t_{i+1}}\mathbb{E}_{t_i}^x[Z_{s}]ds
- \frac{h_t}{2}\mathbb{E}_{t_i}^x[Z_{t_i}+ Z_{t_{i+1}}].
\end{array}
\nonumber
\end{equation}
Plugging (\ref{FBSDE-err-z-1}) and (\ref{FBSDE-err-z-2}) into (\ref{FBSDE-err-z}), we deduce
\begin{equation}%
0= \mathbb{E}_{t_i}^x[Y_{t_{i+1}}\Delta W_{i,i+1}^{\top}]
+h_t\mathbb{E}_{t_i}^x[\frac{1}{2\alpha}f_{t_{i+1-\alpha }}\Delta W_{i,i+1-\alpha}^{\top}+ (1-\frac{1}{2\alpha})f_{t_{i+1}}\Delta W_{i,i+1}^{\top}]
-\frac{h_t}{2}\mathbb{E}_{t_i}^x[Z_{t_i}+ Z_{t_{i+1}}]
+R^i_z,\nonumber
\end{equation}
where $R^i_z = R_{z,1}^i - R_{z,2}^i.$ Furthermore,
\begin{equation}%
Z_{t_i} = \mathbb{E}_{t_i}^x[\frac{2}{h_t}Y_{t_{i+1}}\Delta W_{i,i+1}^{\top}]
+2\mathbb{E}_{t_i}^x[\frac{1}{2\alpha}f_{t_{i+1-\alpha }}\Delta W_{i,i+1-\alpha}^{\top}+ (1-\frac{1}{2\alpha})f_{t_{i+1}}\Delta W_{i,i+1}^{\top}]
-\mathbb{E}_{t_i}^x[ Z_{t_{i+1}}]
+\frac{2}{h_t}R^i_z.\label{FBSDE-z-t-i}
\end{equation}
Analogously, we approximate the value of $Z_{t_{i+1-\alpha}}$ in $f_{t_{i+1-\alpha}}$ by the Euler scheme as below
\begin{align}
\alpha h_t Z_{t_{i+1-\alpha}}
=&\mathbb{E}_{t_{i+1-\alpha}}^x\left[Y_{t_{i+1}}\Delta W^\top_{i+1-\alpha,i+1}
+\alpha h_tf_{t_{i+1}}\Delta W^\top_{i+1-\alpha,i+1}\right] +R^i_{z,3}
\label{E-BSDE-t-i-Euler-1}
\end{align}
where  $\Delta W_{i+1-\alpha,i+1} = W_{t_{i+1}} - W_{t_{i+1-\alpha}},$ $\Delta W_{i+1-\alpha,s} = W_s - W_{t_{i+1-\alpha}},$
\begin{equation}
\begin{array}
[c]{rl}%
&R^i_{z,3}=\mathbb{E}_{t_{i+1-\alpha}}^x\left[R^i_{z,31}-R^i_{z,32}\right],\\
&R^i_{z,31}=
\int_{t_{i+1-\alpha}}^{t_{i+1}}f_s\Delta W^\top_{i+1-\alpha,s}ds
-\alpha h_t f_{t_{i+1}}\Delta W^\top_{i+1-\alpha,i+1}, \\
&R^i_{z,32}=\int_{t_{i+1-\alpha}}^{t_{i+1}}Z_{s}ds
-\alpha h_t Z_{t_{i+1-\alpha}}.
\end{array}
\nonumber
\end{equation}
Thus, the computation expression with respect to $Z_{t_{i+1-\alpha}}$ is, 
\begin{align}
 Z^\pi_{i+1-\alpha}
=&\mathbb{E}_{t_{i+1-\alpha}}^x\left[\frac{1}{\alpha h_t}Y^\pi_{i+1}\Delta W^\top_{i+1-\alpha,i+1}+f_{i+1}^\pi\Delta W^\top_{i+1-\alpha,i+1}\right].\label{E-GBSDE-t-i-pi-1}
\end{align}
Combining (\ref{FBSDE-err-y-1}) with (\ref{FBSDE-err-y-2}), (\ref{E-BSDE-t-i-Euler-1}), we have
\begin{equation}%
Y_{t_i}  = \mathbb{E}_{t_i}^x[Y_{t_{i+1}} +  \frac{h_t}{2\alpha}\widetilde{f}_{t_{i+1-\alpha }}+ h_t(1-\frac{1}{2\alpha})f_{t_{i+1}}] + R^i_y,
\label{FBSDE-err-y-3}
\end{equation}
where $\widetilde{f}_{t_{i+1-\alpha }} =f(t_{i+1-\alpha },\widetilde{Y}_{t_{i+1-\alpha }},\widetilde{Z}_{t_{i+1-\alpha }}),$ for $i=0,1,\cdots,N-1$,
\begin{equation}
\begin{array}
[c]{rl}%
&R^i_y=R^i_{yf}+\widehat{R}^i_{yf}, \widehat{R}^i_{yf}= h_t\mathbb{E}_{t_i}^x[f_{t_{i+1-\alpha}}-\widetilde{f}_{t_{i+1-\alpha}}],\\
&\widetilde{Y}_{t_{i+1-\alpha}}=\mathbb{E}_{t_{i+1-\alpha}}^x[Y_{t_{i+1}}+\alpha h_tf_{t_{i+1}}],\\
&\widetilde{Z}_{t_{i+1-\alpha}}=\mathbb{E}_{t_{i+1-\alpha}}^x[\frac{1}{\alpha h_t}Y_{t_{i+1}}\Delta W^\top_{i+1-\alpha,i+1}+f_{t_{i+1}}\Delta W^\top_{i+1-\alpha,i+1}].
\end{array}
\nonumber
\end{equation}
Thus, the time-discretization of $Y_{t_i}$ is, 
\begin{align}
Y_i^\pi
=&\mathbb{E}_{t_i}^x\left[Y^\pi_{i+1}+  \frac{h_t}{2\alpha}\widetilde{f}^\pi_{i+1-\alpha }+ h_t(1-\frac{1}{2\alpha})f^\pi_{i+1}\right],
\label{E-GBSDE-t-i-M-Euler-pi}
\end{align}
where $\widetilde{f}^\pi_{i+1-\alpha}=f(t_{i+1-\alpha},Y^\pi_{i+1-\alpha},Z^\pi_{i+1-\alpha})$
for $i=0,1,\cdots,N-1$.
From (\ref{FBSDE-z-t-i}), we have the time-discretization of $Z_{t_i}$,
\begin{align}
 Z^\pi_i
=\mathbb{E}_{t_i}^x\left[\frac{2}{h_t}Y^\pi_{i+1}\Delta W_{i,i+1}^{\top}
+\frac{1}{\alpha}\widetilde{f}^\pi_{i+1-\alpha }\Delta W_{i,i+1-\alpha}^{\top}+ \frac{2\alpha-1}{\alpha}f^\pi_{i+1}\Delta W_{i,i+1}^{\top}
-Z^\pi_{i+1}\right].
\label{GBSDE-t-i-z-pi}
\end{align}

Thus, we deduce the  discrete-time
approximation $(Y^{\pi},Z^{\pi})$ for $(Y,Z)$ at $t_{i}$ for $i=N,N-1,\cdots,1,0$
\begin{enumerate}
\item
the terminal condition is $(Y_N^\pi,Z_N^\pi) = (\Phi(X_T), D_x u(T,X_{T})^\top\sigma(T,X_{T}))$,
\item for $0 \leq i < N$, the step from $i+1$ to $i$ follows the transition rule described as
\begin{equation}
\left\{
\begin{aligned}
Y^\pi_{i+1-\alpha}
=&\mathbb{E}_{t_{i+1-\alpha}}^x\left[Y^\pi_{i+1}+\alpha h_tf^\pi_{i+1}\right],\\
 Z^\pi_{i+1-\alpha}
=&\mathbb{E}_{t_{i+1-\alpha}}^x\left[\frac{1}{\alpha h_t}Y^\pi_{i+1}\Delta W^\top_{i+1-\alpha,i+1}+f_{i+1}^\pi\Delta W^\top_{i+1-\alpha,i+1}\right],\\
Y_i^\pi
=&\mathbb{E}_{t_i}^x\left[Y^\pi_{i+1}+ \frac{h_t}{2\alpha}\widetilde{f}^\pi_{i+1-\alpha }+ h_t(1-\frac{1}{2\alpha})f^\pi_{i+1}\right],\\
Z^\pi_i
= & \mathbb{E}_{t_i}^x\left[\frac{2}{h_t}Y^\pi_{i+1}\Delta W_{i,i+1}^{\top}
+\frac{1}{\alpha}\widetilde{f}^\pi_{i+1-\alpha }\Delta W_{i,i+1-\alpha}^{\top}+ \frac{2\alpha-1}{\alpha}f^\pi_{i+1}\Delta W_{i,i+1}^{\top}
-Z^\pi_{i+1}\right].
\end{aligned}
\right.  \label{NumSch}%
\end{equation}
\end{enumerate}

\section{Stability analysis}
In this section, we thoroughly analyze stability for the numerical scheme. To simplify notation, we first focus on a one-dimensional scenario. This approach can be easily generalized to higher dimensions. We denote the perturbations affecting the generator \(f\) and the terminal values \(Y_N^\pi\) and \(Z_N^\pi\) by \(\varepsilon_f\), \(\varepsilon_{y,N}^\pi\), and \(\varepsilon_{z,N}^\pi\), respectively. Note that \(\varepsilon_f\) is an \(\mathcal{F}_t\)-adapted process. At any point \((s, Y, Z) \in \left[0, T\right] \times \mathbb{R} \times \mathbb{R}\), we define \(Y_{\varepsilon,N}^\pi\), \(Z_{\varepsilon,N}^\pi\), and \(f_\varepsilon\) as follows:
\begin{equation}
\left\{
\begin{aligned}
Y_{\varepsilon,N}^\pi &= Y_N^\pi + \varepsilon_{y,N}^\pi, \\
Z_{\varepsilon,N}^\pi &= Z_N^\pi + \varepsilon_{z,N}^\pi, \\
f_\varepsilon(s, Y_s, Z_s) &= f(s, Y_s, Z_s) + \varepsilon_f.
\end{aligned}
\right.  \label{ST-1}%
\end{equation}
Furthermore, we define the following quantities: \( \mathit{f}_{\mathit{i}}^{\varepsilon, \pi } \), \( \widetilde{\mathit{f}}_{{\mathit{i + 1 - }\alpha }}^{\varepsilon, \pi } \), \( \varepsilon _{\mathit{f}}^{{\mathit{i}}, \pi } \), \( \widetilde{\varepsilon} _{\mathit{f}}^{{\mathit{i + 1 - }\alpha}, \pi } \), \( \mathit{f}_{\varepsilon, {\mathit{i}}}^\pi \), and \( \widetilde{\mathit{f}}_{\varepsilon, {\mathit{i}}}^\pi \) as follows:
\begin{align}
 \mathit{f}_{\mathit{i}}^{\varepsilon ,\pi } &= \mathit{f}({{\mathit{t}}_{\mathit{i}}}, Y_{\varepsilon ,{\mathit{i}}}^\pi, Z_{\varepsilon ,{\mathit{i}}}^\pi ), \nonumber \\
 \widetilde{\mathit{f}}_{{\mathit{i + 1 - }}\alpha }^{\varepsilon ,\pi } &= \mathit{f}({{\mathit{t}}_{\mathit{i}}}, Y_{\varepsilon,{\mathit{i + 1 - }}\alpha }^\pi, Z_{\varepsilon,{\mathit{i + 1 - }}\alpha }^\pi ), \nonumber \\
 \varepsilon _{\mathit{f}}^{{\mathit{i}},\pi } &= {\varepsilon _{\mathit{f}}}({{\mathit{t}}_{\mathit{i}}}, Y_{\varepsilon,{\mathit{i}}}^\pi, Z_{\varepsilon,{\mathit{i}}}^\pi), \nonumber \\
 \widetilde{\varepsilon} _{\mathit{f}}^{{\mathit{i + 1 - }}\alpha ,\pi } &= {\varepsilon _{\mathit{f}}}({{\mathit{t}}_{{\mathit{i + 1 - }}\alpha }}, Y_{\varepsilon,{\mathit{i + 1 - }}\alpha }^\pi, Z_{\varepsilon,{\mathit{i + 1 - }}\alpha }^\pi), \nonumber \\
 \mathit{f}_{\varepsilon,{\mathit{i}}}^\pi &= \mathit{f}_{\mathit{i}}^{\varepsilon ,\pi } + \varepsilon _{\mathit{f}}^{{\mathit{i}},\pi }, \quad
 \widetilde{\mathit{f}}_{\varepsilon,{\mathit{i}}}^\pi = \widetilde{\mathit{f}}_{{\mathit{i + 1 - }}\alpha }^{\varepsilon ,\pi } + \widetilde{\varepsilon} _{\mathit{f}}^{{\mathit{i + 1 - }}\alpha ,\pi } . \nonumber
\end{align}
Here $Y_{\varepsilon ,{\mathit{i}}}^\pi$, $Z_{\varepsilon ,{\mathit{i}}}^\pi$, $Y_{\varepsilon,{\mathit{i + 1 - \alpha}}}^\pi$, and $Z_{\varepsilon,{\mathit{i + 1 - \alpha}}}^\pi$ are obtained by the \ref{scheme}.

\begin{flushleft}
\refstepcounter{scheme}
\textbf{\thescheme} Given $(Y_{\varepsilon,N}^\pi, Z_{\varepsilon,N}^\pi  ) = (Y_N^\pi + \varepsilon_{y,N}^\pi, Z_N^\pi + \varepsilon_{z,N}^\pi)$, for $i = N - 1, \dots, 0$, solve $(Y_{\varepsilon,i}^\pi, Z_{\varepsilon,i}^\pi  )$ by
\label{scheme}
\end{flushleft}
\begin{equation}
\left\{
\begin{aligned}
Y^\pi_{\varepsilon,{\mathit{i+1-\alpha}}}
=&\mathbb{E}_{t_{i+1-\alpha}}^x\left[Y^\pi_{\varepsilon,i+1}+\alpha h_tf^\pi_{\varepsilon,i+1}\right],\\
 Z^\pi_{\varepsilon,i+1-\alpha}
=&\mathbb{E}_{t_{i+1-\alpha}}^x\left[\frac{1}{\alpha h_t}Y^\pi_{\varepsilon,i+1}\Delta W^\top_{i+1-\alpha,i+1}+f_{\varepsilon,i+1}^\pi\Delta W^\top_{i+1-\alpha,i+1}\right],\\
Y_{\varepsilon,i}^\pi
=&\mathbb{E}_{t_i}^x\left[Y^\pi_{\varepsilon,i+1}+ \frac{h_t}{2\alpha}\widetilde{f}^\pi_{\varepsilon,i+1-\alpha }+ h_t(1-\frac{1}{2\alpha})f^\pi_{\varepsilon,i+1}\right],\\
Z^\pi_{\varepsilon,i}
= & \mathbb{E}_{t_i}^x\left[\frac{2}{h_t}Y^\pi_{\varepsilon,i+1}\Delta W_{i,i+1}^{\top}
+\frac{1}{\alpha}\widetilde{f}^\pi_{\varepsilon,i+1-\alpha }\Delta W_{i,i+1-\alpha}^{\top}+ \frac{2\alpha-1}{\alpha}f^\pi_{\varepsilon,i+1}\Delta W_{i,i+1}^{\top}
-Z^\pi_{\varepsilon,i+1}\right].
\end{aligned}
\right.  \label{ST-2}%
\end{equation}

To enable further analysis, the above scheme can be equivalently written as
\begin{equation}
\left\{
\begin{aligned}
Y^\pi_{\varepsilon,{\mathit{i+1-\alpha}}}
=&\mathbb{E}_{t_{i+1-\alpha}}^x\left[Y^\pi_{\varepsilon,i+1}+\alpha h_tf^{\varepsilon,\pi}_{i+1}\right]+\widetilde{R}^{\pi}_{\varepsilon y,i },\\
 Z^\pi_{\varepsilon,i+1-\alpha}
=&\mathbb{E}_{t_{i+1-\alpha}}^x\left[\frac{1}{\alpha h_t}Y^\pi_{\varepsilon,i+1}\Delta W^\top_{i+1-\alpha,i+1}+f^{\varepsilon,\pi}_{i+1}\Delta W^\top_{i+1-\alpha,i+1}\right]+\widetilde{R}^{\pi}_{\varepsilon z,i },\\
Y_{\varepsilon,i}^\pi
=&\mathbb{E}_{t_i}^x\left[Y^\pi_{\varepsilon,i+1}+ \frac{h_t}{2\alpha}\widetilde{f}^{\varepsilon,\pi}_{i+1-\alpha }+ h_t(1-\frac{1}{2\alpha})f^{\varepsilon,\pi}_{i+1}\right]+{R}^{\pi}_{\varepsilon y,i },\\
Z^\pi_{\varepsilon,i}
= & \mathbb{E}_{t_i}^x\left[\frac{2}{h_t}Y^\pi_{\varepsilon,i+1}\Delta W_{i,i+1}^{\top}
+\frac{1}{\alpha}\widetilde{f}^{\varepsilon,\pi}_{i+1-\alpha }\Delta W_{i,i+1-\alpha}^{\top}+ \frac{2\alpha-1}{\alpha}f^{\varepsilon,\pi}_{i+1}\Delta W_{i,i+1}^{\top}
-Z^\pi_{\varepsilon,i+1}\right]+{R}^{\pi}_{\varepsilon z,i },
\end{aligned}
\right.  \label{ST-3}%
\end{equation}
where
\begin{equation*}
\left\{
\begin{aligned}
\widetilde{R}^{\pi}_{\varepsilon y,i }
=&\mathbb{E}_{t_{i+1-\alpha}}^x\left[\alpha h_t\varepsilon _{\mathit{f}}^{{i+1},\pi }\right],\\
 \widetilde{R}^{\pi}_{\varepsilon z,i }
=&\mathbb{E}_{t_{i+1-\alpha}}^x\left[\varepsilon _{\mathit{f}}^{{i+1},\pi}\Delta W^\top_{i+1-\alpha,i+1}\right],\\
{R}^{\pi}_{\varepsilon y,i }
=&\mathbb{E}_{t_i}^x\left[ \frac{h_t}{2\alpha}\widetilde{\varepsilon} _{\mathit{f}}^{{i+1-\alpha},\pi }+ h_t(1-\frac{1}{2\alpha})\varepsilon_{\mathit{f}}^{{i+1},\pi }\right],\\
{R}^{\pi}_{\varepsilon z,i }
= & \mathbb{E}_{t_i}^x\left[
\frac{1}{\alpha}\widetilde{\varepsilon} _{\mathit{f}}^{{i+1-\alpha},\pi }\Delta W_{i,i+1-\alpha}^{\top}+ \frac{2\alpha-1}{\alpha}\varepsilon_{\mathit{f}}^{{i+1},\pi }\Delta W_{i,i+1}^{\top}\right].
\end{aligned}
\right.  
\end{equation*}

To obtain the error analysis expression, we let $\varepsilon_{y,N}^{\pi}=Y_{\varepsilon,N}^\pi-Y_{N}^\pi$, $\varepsilon_{z,N}^{\pi}=Z_{\varepsilon,N}^\pi-Z_{N}^\pi$. Then subtracting (\ref{NumSch}) from (\ref{ST-3}), we obtain
\begin{equation}
\left\{
\begin{aligned}
\varepsilon_{y,i+1-\alpha}^{\pi}
=&\mathbb{E}_{t_{i+1-\alpha}}^x\left[\varepsilon_{y,i+1}^{\pi}+\alpha h_t(f^{\varepsilon,\pi}_{i+1}-f^{\pi}_{i+1})\right]+\widetilde{R}^{\pi}_{\varepsilon y,i },\\
 \varepsilon_{z,i+1-\alpha}^{\pi}
=&\mathbb{E}_{t_{i+1-\alpha}}^x\left[\frac{1}{\alpha h_t}\varepsilon_{y,i+1}^{\pi}\Delta W^\top_{i+1-\alpha,i+1}+(f^{\varepsilon,\pi}_{i+1}-f^{\pi}_{i+1})\Delta W^\top_{i+1-\alpha,i+1}\right]+\widetilde{R}^{\pi}_{\varepsilon z,i },\\
\varepsilon_{y,i}^{\pi}
=&\mathbb{E}_{t_i}^x\left[\varepsilon_{y,i+1}^{\pi}+ \frac{h_t}{2\alpha}(\widetilde{f}^{\varepsilon,\pi}_{i+1-\alpha }-\widetilde{f}^{\pi}_{i+1-\alpha })+ h_t(1-\frac{1}{2\alpha})(f^{\varepsilon,\pi}_{i+1}-f^{\pi}_{i+1})\right]+{R}^{\pi}_{\varepsilon y,i },\\
\varepsilon_{z,i}^{\pi}
= & \mathbb{E}_{t_i}^x\left[\frac{2}{h_t}\varepsilon_{y,i+1}^{\pi}\Delta W_{i,i+1}^{\top}
+\frac{1}{\alpha}(\widetilde{f}^{\varepsilon,\pi}_{i+1-\alpha }-\widetilde{f}^{\pi}_{i+1-\alpha })\Delta W_{i,i+1-\alpha}^{\top}+ \frac{2\alpha-1}{\alpha}(f^{\varepsilon,\pi}_{i+1}-f^{\pi}_{i+1})\Delta W_{i,i+1}^{\top}
\right]+{R}^{\pi}_{\varepsilon z,i }.
\end{aligned}
\right.  \label{ST-5}%
\end{equation}
We refer to (\ref{ST-5}) as the permutation error equations for scheme (\ref{NumSch}). To prove the stability of scheme (\ref{NumSch}), stability is defined as follows.

\begin{flushleft}
\refstepcounter{definition}
\textbf{\thedefinition} The scheme (\ref{NumSch}) is defined as stable. If, for any $\varepsilon > 0$ and $0 \leq i \leq N-1$, there exists a $\delta > 0$, such that
\[
\mathbb{E}\left[ |\varepsilon_{y,i}^{\pi}|^2 + h_t \sum_{\ell=i}^{N-1} |\varepsilon_{z,\ell}^{\pi}|^2 \right] < \varepsilon,
\]
provided that
\[
\mathbb{E}\left[ |\varepsilon_{\mathit{f}}^{i,\pi}|^2 + |\widetilde{\varepsilon} _{\mathit{f}}^{{i+1-\alpha},\pi }|^2 \right] < \delta,
\mathbb{E}\left[ |\varepsilon_{y,N}^{\pi}|^2 + |\varepsilon_{z,N}^{\pi}|^2 \right] < \delta.
\]
\label{definition}
\end{flushleft}
\begin{theorem}\label{ST-theorem}
Suppose the assumptions (i)-(iii) and $f(t,Y_t,Z_t)\in C_b^3$ hold. Then for a sufficiently small step size $h_t$, we have
\begin{equation}
\begin{aligned}
\mathbb{E}\left[ |\varepsilon_{y,i}^{\pi}|^2 \right]
+ \mathbb{E}\left[ h_t\sum_{\ell=i}^{N-1} |\varepsilon_{z,\ell}^{\pi}|^2 \right]
\le&
C \left(\mathbb{E}[|\varepsilon_{y,N}^{\pi}|^2]
+h_t\mathbb{E}[|\varepsilon_{z,N}^{\pi}|^2] \right)\nonumber\\
& + \sum_{\ell=i}^{N-1}C\mathbb{E}\left[
\left|R^{\pi}_{\varepsilon y,\ell}\right|^2 + (h_t +  h_t^2)\left(\left|\widetilde{R}^{\pi}_{\varepsilon y,\ell}\right|^2 + \left|\widetilde{R}^{\pi}_{\varepsilon z,\ell}\right|^2\right)
+ h_t \left|R^{\pi}_{\varepsilon z,\ell}\right|^2 \right],
\end{aligned}
\end{equation}
where 
$C$ denotes a positive constant.
\end{theorem}

\begin{proof}
From (\ref{ST-5}) and the Lipschitz-condition assumption, we have
\begin{equation}
\begin{aligned}
\left|\varepsilon_{z,i}^{\pi}\right|
\leq & \ \left|\mathbb{E}_{t_i}^x\left[\varepsilon_{y,i+1}^{\pi} \frac{2 \Delta W_{i,i + 1}^\top}{h_t} \right] \right|
+ \frac{L}{\alpha} \mathbb{E}_{t_i}^x\left[ \left( |\varepsilon_{y,i+1-\alpha}^{\pi}| + |\varepsilon_{z,i+1-\alpha}^{\pi}| \right) |\Delta W_{i,i + 1 - \alpha}^\top| \right] \\
& + \left| \frac{2\alpha - 1}{\alpha} \right| L \mathbb{E}_{t_i}^x\left[\left( |\varepsilon_{y,i+1}^{\pi}| + |\varepsilon_{z,i+1}^{\pi}| \right) |\Delta W_{i,i + 1 - \alpha}^\top|\right] + \mathbb{E}_{t_i}^x\left[|\varepsilon_{z,i+1}^{\pi}| \right] + \left| R_{\varepsilon z,i}^{\pi} \right|.
\end{aligned}
\label{ST-6}
\end{equation}
From the Cauchy-Schwarz inequality, the inequality (\ref{ST-6}) is represented as
\begin{equation}
\begin{aligned}
\left|\varepsilon_{z,i}^{\pi}\right|
\leq & \ \left| \frac{2}{h_t} \sqrt{h_t \mathbb{E}_{t_i}^x[|\varepsilon_{y,i+1}^{\pi}|^2]} \right|
+ \frac{L}{\alpha} \left( \sqrt{(1-\alpha)h_t \mathbb{E}_{t_i}^x[|\varepsilon_{y,i + 1 - \alpha}^{\pi}|^2]} + \sqrt{(1-\alpha)h_t \mathbb{E}_{t_i}^x[|\varepsilon_{z,i + 1 - \alpha}^{\pi}|^2]} \right)  \\
& + \left| \frac{2\alpha - 1}{\alpha} \right| L \left( \sqrt{h_t \mathbb{E}_{t_i}^x[|\varepsilon_{y,i + 1}^{\pi}|^2]} + \sqrt{h_t \mathbb{E}_{t_i}^x[|\varepsilon_{z,i + 1}^{\pi}|^2]} \right)   + \mathbb{E}_{t_i}^x\left[|\varepsilon_{z,i+1}^{\pi}| \right] + \left| R_{\varepsilon z,i}^{\pi} \right|.
\end{aligned}
\label{ST-7}
\end{equation}
From (\ref{ST-5}), we have
\begin{equation}
\begin{aligned}
|\varepsilon_{y,i+1-\alpha}^{\pi}|^2
=&\left|\mathbb{E}_{t_{i+1-\alpha}}^x\left[\varepsilon_{y,i+1}^{\pi}+\alpha h_t(f^{\varepsilon,\pi}_{i+1}-f^{\pi}_{i+1})\right]+\widetilde{R}^{\pi}_{\varepsilon y,i }\right|^2\\
\le&\left|\mathbb{E}_{t_{i+1-\alpha}}^x\left[\varepsilon_{y,i+1}^{\pi}+\alpha h_t L(|\varepsilon_{y,i+1}^{\pi}| + |\varepsilon_{z,i+1}^{\pi}|)\right]+\widetilde{R}^{\pi}_{\varepsilon y,i }\right|^2,
\end{aligned}
\label{ST-8}
\end{equation}
and
\begin{equation}
\begin{aligned}
|\varepsilon_{z,i+1-\alpha}^{\pi}|^2
=&\left|\mathbb{E}_{t_{i+1-\alpha}}^x\left[\frac{1}{\alpha h_t}|\varepsilon_{y,i+1}^{\pi}|\Delta W^\top_{i+1-\alpha,i+1}+(f^{\varepsilon,\pi}_{i+1}-f^{\pi}_{i+1})\Delta W^\top_{i+1-\alpha,i+1}\right]+\widetilde{R}^{\pi}_{\varepsilon z,i }\right|^2\\
\le& \left|\mathbb{E}_{t_{i + 1 - \alpha}}^x\left[\frac{1}{\alpha h_t}|\varepsilon_{y,i+1}^{\pi}|\Delta W^\top_{i + 1 - \alpha,i+1}
+L\left(|\varepsilon_{y,i + 1}^{\pi}| + |\varepsilon_{z,i+1}^{\pi}|\right)\Delta W^\top_{i + 1 - \alpha,i+1}\right] +\widetilde{R}^{\pi}_{\varepsilon z,i }\right|^2\\
\le& 2\mathbb{E}_{t_{i + 1 - \alpha}}^x\left[\left(\frac{2}{\alpha h_t}+2\alpha h_tL^2+4L\right)|\varepsilon_{y,i+1}^{\pi}|^2+2\alpha h_tL^2|\varepsilon_{z,i+1}^{\pi}|^2\right] +2\left|\widetilde{R}^{\pi}_{\varepsilon z,i }\right|^2.
\end{aligned}
\label{ST-9}
\end{equation}
Plugging (\ref{ST-8}) and (\ref{ST-9}) into (\ref{ST-7}), we obtain
\begin{align}
\left|\varepsilon_{z,i}^{\pi}\right|
\leq & \ \left| \frac{2}{h_t} \sqrt{h_t \mathbb{E}_{t_i}^x[|\varepsilon_{y,i+1}^{\pi}|^2]} \right|
+ \frac{L}{\alpha} \left( \sqrt{2(1-\alpha)h_t \mathbb{E}_{t_i}^x\left[\left( (1 + \alpha h_t L) |\varepsilon_{y,i + 1}^{\pi}| + \alpha h_tL|\varepsilon_{z,i + 1}^{\pi}| \right)^2 + (\widetilde{R}^{\pi}_{\varepsilon y,i })^2 \right]} \right) \nonumber\\
& + \frac{L}{\alpha} \sqrt{2(1-\alpha)h_t \mathbb{E}_{t_i}^x\left[\left( \frac{2}{\alpha h_t} + 2\alpha h_t L^2 + 4L \right) |\varepsilon_{y,i + 1}^{\pi}|^2 + 2\alpha h_t L^2 |\varepsilon_{z,i + 1}^{\pi}|^2 + |\widetilde{R}^{\pi}_{\varepsilon z,i }|^2 \right]} \nonumber\\
& + \left| \frac{2\alpha - 1}{\alpha} \right| L \left( \sqrt{h_t \mathbb{E}_{t_i}^x[|\varepsilon_{y,i + 1}^{\pi}|^2]} + \sqrt{h_t \mathbb{E}_{t_i}^x[|\varepsilon_{z,i + 1}^{\pi}|^2]} \right)  + \mathbb{E}_{t_i}^x[|\varepsilon_{z,i + 1}^{\pi}|] + \left|  R_{\varepsilon z,i}^{\pi} \right|.
\label{ST-10}
\end{align}
Squaring (\ref{ST-10}), we deduce
\begin{align}
\left|\varepsilon_{z,i}^{\pi}\right|^2
\le & \, \left(\frac{20}{h_t} + \frac{10(1 - \alpha)h_t L^2}{\alpha^2}
\left(4 + 4\alpha^2 h_t L^2 + 8\alpha h_t L + \frac{2}{\alpha h_t} + 2\alpha h_t L^2 + 4L\right) \right. \nonumber \\
& \left. + \frac{10{(2\alpha - 1)}^2 h_t L^2}{\alpha^2}\right)\mathbb{E}_{t_i}^x[|\varepsilon_{y,i+1}^{\pi}|^2] \nonumber \\
& + \left(\frac{10(1 - \alpha)h_t L^2}{\alpha^2}(4\alpha^2 h_t^2 L^2 + 2\alpha h_t L^2)
+ \frac{10{(2\alpha - 1)}^2 h_t L^2}{\alpha^2} + 5 \right)\mathbb{E}_{t_i}^x[|\varepsilon_{z,i+1}^{\pi}|^2] \nonumber \\
& + \frac{20(1 - \alpha)h_t L^2}{\alpha^2}\left(|\widetilde{R}^{\pi}_{\varepsilon z,i }|^2 + |\widetilde{R}^{\pi}_{\varepsilon y,i }|^2\right)
+ 5\left|R_{\varepsilon z,i}^{\pi}\right|^2.
\label{ST-11}
\end{align}
From (\ref{ST-5}), we preliminarily obtain
\begin{align}
\left|\varepsilon_{y,i}^{\pi}\right|
\le & \left| \mathbb{E}_{t_i}^x[\varepsilon_{y,i+1}^{\pi}] \right|
+ \left| \frac{h_t}{2\alpha} \mathbb{E}_{t_i}^x[\widetilde{f}^{\varepsilon,\pi}_{i+1-\alpha }-\widetilde{f}^{\pi}_{i+1-\alpha }] \right|
+ \left| \frac{h_t(2\alpha - 1)}{2\alpha} \mathbb{E}_{t_i}^x[f^{\varepsilon,\pi}_{i+1}-f^{\pi}_{i+1}] \right|
+ \left| {R}^{\pi}_{\varepsilon y,i } \right| \nonumber \\
\le & \left| \mathbb{E}_{t_i}^x[\varepsilon_{y,i+1}^{\pi}] \right|
+ \frac{h_t}{2\alpha} L \mathbb{E}_{t_i}^x \left[ |\varepsilon_{y,i+1-\alpha}^{\pi}| + |\varepsilon_{z,i+1-\alpha}^{\pi}| \right]  + \left|\frac{2\alpha - 1}{2\alpha}\right| h_t L \mathbb{E}_{t_i}^x \left[ \left| \varepsilon_{y,i+1}^{\pi} \right| + \left| \varepsilon_{z,i+1}^{\pi} \right| \right]
+ \left| {R}^{\pi}_{\varepsilon y,i } \right|.
\label{ST-12}
\end{align}
Combining (\ref{ST-5}) and (\ref{ST-12}), we can further derive
\begin{align}
\left|\varepsilon_{y,i}^{\pi}\right|
\le & \left(1 + \frac{h_t L}{2\alpha} (1 + \alpha h_t L) + \left| \frac{2\alpha - 1}{\alpha} \right| h_t L\right) \mathbb{E}_{t_i}^x \left[\left|\varepsilon_{y,i+1}^{\pi}\right|\right] \nonumber \\
& + \left( \frac{h_t^2 L^2}{\alpha^2} + \left| \frac{2\alpha - 1}{\alpha} \right| h_t L \right) \mathbb{E}_{t_i}^x \left[\left|\varepsilon_{z,i+1}^{\pi}\right|\right]
+ \left( \frac{L}{2\alpha^2} + \frac{h_t L^2}{2\alpha} \right) \mathbb{E}_{t_i}^x \left[\left|\varepsilon_{y,i+1}^{\pi} \Delta W^\top_{i + 1 - \alpha, i + 1}\right|\right] \nonumber \\
& + \frac{h_t L^2}{2\alpha} \mathbb{E}_{t_i}^x \left[\left|\varepsilon_{z,i+1}^{\pi} \Delta W^\top_{i + 1 - \alpha, i + 1}\right|\right]
+ \left|R^{\pi}_{\varepsilon y,i}\right|+\frac{h_t L}{2\alpha}\left|\widetilde{R}^{\pi}_{\varepsilon y,i}\right|+\frac{h_t L}{2\alpha}\left|\widetilde{R}^{\pi}_{\varepsilon z,i}\right|.
\label{ST-13}
\end{align}
Squaring (\ref{ST-13}), we deduce
\begin{align}
\left|\varepsilon_{y,i}^{\pi}\right|^2
\le& \, 5 \left( 4 + \frac{h_t^2 L^2}{\alpha^2} + h_t^4 L^4 + \frac{4(2\alpha - 1)^2}{\alpha^2} h_t^2 L^2 + \frac{h_t L^2}{2 \alpha^3} + \frac{h_t^3 L^4}{2 \alpha} \right) \mathbb{E}_{t_i}^x \left[|\varepsilon_{y,i+1}^{\pi}|^2\right] \nonumber \\
&+ 5 \left( \frac{h_t^4 L^4}{2} + \frac{2(2\alpha - 1)^2}{\alpha^2} h_t^2 L^2 + \frac{h_t^3 L^4}{4 \alpha} \right) \mathbb{E}_{t_i}^x \left[|\varepsilon_{z,i+1}^{\pi}|^2\right]\nonumber \\
&+ 15\left|R^{\pi}_{\varepsilon y,i}\right|^2+\frac{15h_t L}{2\alpha}\left|\widetilde{R}^{\pi}_{\varepsilon y,i}\right|^2+\frac{15h_t L}{2\alpha}\left|\widetilde{R}^{\pi}_{\varepsilon z,i}\right|^2.
\label{ST-14}
\end{align}
Adding (\ref{ST-11}) to (\ref{ST-14}), we deduce
\begin{align}
&\left|\varepsilon_{y,i}^{\pi}\right|^2 + h_t\left|\varepsilon_{z,i}^{\pi}\right|^2 \nonumber\\
\le& \, 5 \bigg( 8 + \frac{4(1 - \alpha)}{\alpha^3} L^2 h_t + \frac{6(2\alpha - 1)^2 + 1 + 8(1 - \alpha) + 8L(1 - \alpha)}{\alpha^2} L^2 h_t^2 \nonumber\\
& \quad + \frac{L + 32(1 - \alpha) + 8L(1 - \alpha)}{2\alpha} L^3 h_t^3 + (1 + 8(1 - \alpha)) L^4 h_t^4 \bigg) \mathbb{E}_{t_i}^x \left[|\varepsilon_{y,i+1}^{\pi}|^2\right] \nonumber\\
& + 5 \left( h_t + \frac{4(2\alpha - 1)^2}{\alpha^2} L^2 h_t^2 + \frac{16(1 - \alpha) + 1}{4\alpha} L^4 h_t^3 + \frac{16(1 - \alpha) + 1}{2} L^4 h_t^4 \right) \mathbb{E}_{t_i}^x \left[|\varepsilon_{z,i+1}^{\pi}|^2 \right] \nonumber\\
& +15\left|R^{\pi}_{\varepsilon y,i}\right|^2+\frac{15h_t L}{2\alpha}\left|\widetilde{R}^{\pi}_{\varepsilon y,i}\right|^2+\frac{15h_t L}{2\alpha}\left|\widetilde{R}^{\pi}_{\varepsilon z,i}\right|^2 + \frac{20(1 - \alpha)h_t^2 L^2}{\alpha^2}\left(|\widetilde{R}^{\pi}_{\varepsilon z,i }|^2 + |\widetilde{R}^{\pi}_{\varepsilon y,i }|^2\right)
+ 5h_t\left|R_{\varepsilon z,i}^{\pi}\right|^2 \nonumber\\
=& \, 40(1 + C_1 h_t) \mathbb{E}_{t_i}^x \left[|\varepsilon_{y,i+1}^{\pi}|^2\right] + C_2 h_t \mathbb{E}_{t_i}^x \left[|\varepsilon_{z,i+1}^{\pi}|^2 \right] \nonumber\\
& +15\left|R^{\pi}_{\varepsilon y,i}\right|^2+\left(\frac{15h_t L}{2\alpha} + \frac{20(1 - \alpha)h_t^2 L^2}{\alpha^2}\right)\left(|\widetilde{R}^{\pi}_{\varepsilon z,i }|^2 + |\widetilde{R}^{\pi}_{\varepsilon y,i }|^2\right)
+ 5h_t\left|R_{\varepsilon z,i}^{\pi}\right|^2,
\label{ST-15}
\end{align}
where $C_1 = \frac{1}{8}(\frac{{4(1 - \alpha )}}{{{\alpha ^3}}}{L^2} + \frac{{6{{(2\alpha  - 1)}^2} + 1 + 8(1 - \alpha ) + 8L(1 - \alpha )}}{{{\alpha ^2}}}{L^2}{h_t} + \frac{{L + 32(1 - \alpha ) + 8L(1 - \alpha )}}{{2\alpha }}{L^3}{h_t}^2 + (1 + 8(1 - \alpha )){L^4}{h_t}^3)$, $C_2 =5(1 + \frac{{4{{(2\alpha  - 1)}^2}}}{{{\alpha ^2}}}{L^2}{h_t} + \frac{{16(1 - \alpha ) + 1}}{{4\alpha }}{L^4}{h_t}^2 + \frac{{16(1 - \alpha ) + 1}}{2}{L^4}{h_t}^3).$
Two positive constants, denoted as $C_0$ and $C_3$, exist such that
\begin{align}
&C_0\left|\varepsilon_{y,i}^{\pi}\right|^2
+ h_t\left|\varepsilon_{z,i}^{\pi}\right|^2\nonumber\\
\le&
(1+C_3h_t)\left(C_0\mathbb{E}_{t_i}^x[|\varepsilon_{y,i+1}^{\pi}|^2]
+C_2h_t\mathbb{E}_{t_i}^x[|\varepsilon_{z,i+1}^{\pi}|^2 ]\right)\nonumber\\
& +C_3\left|R^{\pi}_{\varepsilon y,i}\right|^2+(C_3 h_t+C_3 h_t^2)\left(|\widetilde{R}^{\pi}_{\varepsilon z,i }|^2 + |\widetilde{R}^{\pi}_{\varepsilon y,i }|^2\right)
+ C_3 h_t\left|R_{\varepsilon z,i}^{\pi}\right|^2.
\label{ST-16}
\end{align}
By applying the tower property of conditional expectations to (\ref{ST-16}), it follows that
\begin{align}
&C_0 \mathbb{E}\left[\left|\varepsilon_{y,i}^{\pi}\right|^2\right]
+ (1 - C_2)h_t \sum_{\ell=i}^{N-1} (1 + C_3 h_t)^{\ell-i} \mathbb{E}\left[\left|\varepsilon_{z,\ell}^{\pi}\right|^2\right] \nonumber \\
\le & (1 + C_3 h_t)^{N-i} \left(C_0 \mathbb{E}\left[|\varepsilon_{y,N}^{\pi}|^2\right]
+ C_2 h_t \mathbb{E}\left[|\varepsilon_{z,N}^{\pi}|^2\right]\right) \nonumber \\
& + \sum_{\ell=i}^{N-1} (1 + C_3 h_t)^{\ell-i} \mathbb{E}\left[
C_3 \left|R^{\pi}_{\varepsilon y,\ell}\right|^2
+ \left(C_3 h_t+C_3 h_t^2 \right) \left(\left|\widetilde{R}^{\pi}_{\varepsilon z,\ell}\right|^2 + \left|\widetilde{R}^{\pi}_{\varepsilon y,\ell}\right|^2\right)
+ C_3 h_t \left|R^{\pi}_{\varepsilon z,\ell}\right|^2\right].
\label{ST-17}
\end{align}
Furthermore, we have
\begin{align}
\mathbb{E}\left[ |\varepsilon_{y,i}^{\pi}|^2 \right]
+ \mathbb{E}\left[ h_t\sum_{\ell=i}^{N-1} |\varepsilon_{z,\ell}^{\pi}|^2 \right]
\le&
C \left(\mathbb{E}[|\varepsilon_{y,N}^{\pi}|^2]
+h_t\mathbb{E}[|\varepsilon_{z,N}^{\pi}|^2] \right)\nonumber\\
& + \sum_{\ell=i}^{N-1}C\mathbb{E}\left[
\left|R^{\pi}_{\varepsilon y,\ell}\right|^2 + (h_t +  h_t^2)\left(\left|\widetilde{R}^{\pi}_{\varepsilon y,\ell}\right|^2 + \left|\widetilde{R}^{\pi}_{\varepsilon z,\ell}\right|^2\right)
+ h_t \left|R^{\pi}_{\varepsilon z,\ell}\right|^2 \right].
\label{ST-18}
\end{align}
The proof is completed.
\end{proof}

\begin{flushleft}
\refstepcounter{remark}
\textbf{\theremark} 
In stability theory, the impact of small perturbations to the terminal conditions and the generator of BSDE is analyzed to assess the stability of the numerical scheme under these disturbances. Specifically, perturbations $(\varepsilon_{y,N}^\pi, \varepsilon_{z,N}^\pi)$ are applied to the terminal values $(Y_N^\pi, Z_N^\pi)$, and a perturbation $\varepsilon_f$ is introduced to the generator $f(s, Y_s, Z_s)$. This stability analysis confirms the robustness of the scheme described in (\ref{NumSch}) for practical applications.

\end{flushleft}

\section{Error estimates}
In this section, prior to presenting the error estimates for the scheme (\ref{NumSch}), we first introduce a key proposition and three lemmas. To facilitate understanding, we define certain symbols before discussing the It\^{o}-Taylor expansions as outlined in Theorem 5.5.1 of \cite{KPEPE92}. Let \(\alpha := (\alpha_1, \cdots, \alpha_p)\) be a multi-index of finite length, where \(p \in \mathbb{N}^+\).
Define \(\ell(\alpha)\) as the length of the multi-index \(\alpha\). The set \(\mathcal{A}^{\alpha}\) includes all functions \(v: [0, T] \times \mathbb{R}^d \to \mathbb{R}^{d_1}\) such that \(\mathcal{L}^\alpha v\) is continuous and well-defined. Here, \(\mathcal{L}^\alpha = \mathcal{L}^{(\alpha_1)} \circ \cdots \circ \mathcal{L}^{(\alpha_p)}\), and \(\mathcal{L}^{(\alpha_\ell)}\) is expressed as \(\sum_{k=1}^d \sigma_{kj} \partial_{x_k}\) for \(\ell = 1, 2, \cdots, p\). The subset \(\mathcal{A}^{\alpha}_{b}\) includes functions in \(\mathcal{A}^{\alpha}\) where \(\mathcal{L}^\alpha v\) is bounded. For a positive integer \(m\), \(\mathcal{A}^{m}_{b}\) denotes the set of functions \(v\) such that \(v^\alpha \in \mathcal{A}^{\alpha}_{b}\) for all \(\alpha \in \{\alpha \mid \ell(\alpha) \leq m\} \setminus \{\oslash\}\).

\begin{proposition}
(see \cite{CJF14,KPEPE92,ZCWJZW19})
Let $m\ge0$. Then for a function $v\in\mathcal{A}_{b}^{m+1},$
\begin{equation}\label{Ito-Taylor-expansions}
\mathbb{E}_{t}^x[v(t+h_t,X_{t+h_t})]=v_{t}+h_tv_{t}^{(0)}+\frac{h_t^{2}}{2}v_{t}%
^{(0,0)}+\cdots+\frac{h_t^{m}}{m!}v_{t}^{(0)_{m}}+O(h_t^{m+1}),
\end{equation}
\begin{equation}\label{Ito-Taylor-expansions-dw}
\mathbb{E}_{t}^x[(W_{t+h_t}-W_t)v(t+h_t,X_{t+h_t})]=h_tv_{t}^{(1)}+h_t^2v_{t}^{(1,0)}+\frac{h_t^{3}}{2}v_{t}%
^{(1)*(0,0)}+\cdots+\frac{h_t^{m+1}}{m!}v_{t}^{(1)*(0)_{m}}+O(h_t^{m+2}),
\end{equation}
where
$v_{t}=v(t,X_t),$
$v_t^{(0)}=\mathcal{L}^{(0)}v(t,X_t),v_t^{(0,0)}=\mathcal{L}^{(0)}\circ\mathcal{L}^{(0)}v(t,X_t),\cdots,$
$v_t^{(0)_m}=\begin{matrix} \underbrace{\mathcal{L}^{(0)}\circ\cdots\circ\mathcal{L}^{(0)}}_{m} \end{matrix}v(t,X_t);$
$v_t^{(1)}=\mathcal{L}^{(1)}v(t,X_t),v_t^{(1,0)}=\mathcal{L}^{(1)}\circ\mathcal{L}^{(0)}v(t,X_t),\cdots,$
$v_t^{(1)*(0)_m}=\begin{matrix}\mathcal{L}^{(1)}\circ\underbrace{\mathcal{L}^{(0)}\circ\cdots\circ\mathcal{L}^{(0)}}_{m} \end{matrix}v(t,X_t)$.

\end{proposition}

\begin{lemma} \label{parametersofY}
Suppose the assumptions (i)-(iii) hold. Furthermore let
$f(t,Y_t,Z_t)\in C_b^3$.
Then, for $i=0,1,\cdots,N-1$,
$$R^i_{yf}= O(h_t^{3}),\quad\widetilde{R}^i_{yf}= O(h_t^{2}).$$
\end{lemma}

\begin{proof}
By the assumptions (i)-(iii), the integrand $\mathbb{E}_{t}^x[f_{s}]$, for $s > t$, remains continuous in $s$. Consequently, we arrive at the ordinary differential equation

\begin{equation}
\frac{d\mathbb{E}_{t}^x[Y_{s}]}{ds}=-\mathbb{E}_{t}^x[f_{s}],~~~s\in[t,T]. \label{3-1}%
\end{equation}
From the definition of $R^i_{yf}$ and (\ref{FBSDE-err-y-1}), we know
\begin{equation}
\begin{aligned}
R^i_{yf} = & \int_{t_i}^{t_{i + 1}} \mathbb{E}_{t_i}^x [f_s] \, ds - h_t \mathbb{E}_{t_i}^x \left[\frac{1}{2\alpha} f_{t_{i + 1 - \alpha}} + \left(1 - \frac{1}{2\alpha}\right) f_{t_{i + 1}}\right] \\
= & {Y_{{t_i}}} - \mathbb{E}_{{t_i}}^x\left[ {{Y_{{t_{i + 1}}}} + \frac{{{h_t}}}{{2\alpha }}{f_{{t_{i + 1 - \alpha }}}} + {h_t}(1 - \frac{1}{{2\alpha }}){f_{{t_{i + 1}}}}} \right].
\label{Err-yf}
\end{aligned}
\end{equation}
Substituting (\ref{3-1}) into (\ref{Err-yf}), then utilizing (\ref{2-3}) to the derived equation, we have
\begin{equation}
\begin{aligned}
R^i_{yf} = \mathbb{E}_{t_i}^x \left[ u_{t_i} - u_{t_{i+1}} + \frac{h_t}{2\alpha} u_{t_{i + 1 - \alpha}}^{(0)} + h_t \left(1 - \frac{1}{2\alpha}\right) u_{t_{i + 1}}^{(0)} \right],
\label{Err-yf-1}
\end{aligned}
\end{equation}
where $u_{t_i} = u(t_i,X_{t_i})$.
Applying It\^{o}-Taylor expansion (\ref{Ito-Taylor-expansions}) at $(t_{i},X_{t_i})$, we have
\begin{equation}
\begin{aligned}
\mathbb{E}_{t_i}^x[u_{t_{i+1}}] &= u_{t_i} + h_t u^{(0)}_{t_i} + \frac{h_t^2}{2} u_{t_i}^{(0,0)} + \frac{h_t^3}{3!} u_{t_i}^{(0,0,0)} + O(h_t^4), \\
\mathbb{E}_{t_i}^x[u_{t_{i + 1 - \alpha}}^{(0)}] &= u_{t_i}^{(0)} + (1 - \alpha) h_t u_{t_i}^{(0,0)} + \frac{(1 - \alpha)^2 h_t^2}{2} u_{t_i}^{(0,0,0)} + O(h_t^3), \\
\mathbb{E}_{t_i}^x[u_{t_{i + 1}}^{(0)}] &= u_{t_i}^{(0)} + h_t u_{t_i}^{(0,0)} + \frac{h_t^2}{2} u_{t_i}^{(0,0,0)} + O(h_t^3).
\label{Err-yf-2}
\end{aligned}
\end{equation}
Plugging (\ref{Err-yf-2}) into (\ref{Err-yf-1}), we get
\begin{equation}
\begin{aligned}
R^i_{yf}
= & \mathbb{E}_{{t_i}}^x[\frac{{(3\alpha  - 2)h_t^3}}{{12}}u_{{t_i}}^{(0,0,0)} + O(h_t^4)]. \label{Err-yf-3}
\end{aligned}
\end{equation}
(\ref{Err-yf-3}) shows that the result $R^i_{yf}= O(h_t^{3})$ is obvious.
Analogously, by the definition of $\widetilde{R}^i_{yf}$ and (\ref{FBSDE-err-y-2}), we obtain
\begin{equation}
\begin{aligned}
\widetilde{R}^i_{yf} = & \int_{{t_{i + 1 - \alpha }}}^{{t_{i + 1}}} {\mathbb{E}_{{t_{i + 1 - \alpha }}}^x} [f(s,{Y_s},{Z_s})]ds - \alpha {h_t}\mathbb{E}_{{t_{i + 1 - \alpha }}}^x[{f_{{t_{i + 1}}}}]\\
= & {Y_{{t_{i + 1 - \alpha }}}} - \mathbb{E}_{{t_{i + 1 - \alpha }}}^x[{Y_{{t_{i + 1}}}} + \alpha {h_t}{f_{{t_{i + 1}}}}]. \label{Err-yf-1-1}
\end{aligned}
\end{equation}
Substituting (\ref{3-1}) into (\ref{Err-yf-1-1}) and then using It\^{o}-Taylor expansion (\ref{Ito-Taylor-expansions}) at $(t_{i + 1 - \alpha}, X_{t_{i + 1 - \alpha}})$ on the derived equation, we have
\begin{equation}
\begin{aligned}
\widetilde{R}^i_{yf}
= & \mathbb{E}_{{t_{i + 1 - \alpha }}}^x[{u_{{t_{i + 1 - \alpha }}}} - {u_{{t_{i + 1}}}} + \alpha {h_t}u_{{t_{i + 1}}}^{(0)}] \\
= & \mathbb{E}_{{t_{i + 1 - \alpha }}}^x[{u_{{t_{i + 1 - \alpha }}}} - ({u_{{t_{i + 1 - \alpha }}}} + \alpha {h_t}u_{{t_{i + 1 - \alpha }}}^{(0)} + \frac{{{\alpha ^2}h_t^2}}{2}u_{{t_{i + 1 - \alpha }}}^{(0,0)} + O(h_t^3)) \\
& + \alpha {h_t}(u_{{t_{i + 1 - \alpha }}}^{(0)} + \alpha {h_t}u_{{t_{i + 1 - \alpha }}}^{(0,0)} + O(h_t^2))] \\
= & \mathbb{E}_{{t_{i + 1 - \alpha }}}^x[\frac{{{\alpha ^2}h_t^2}}{2}u_{{t_{i + 1 - \alpha }}}^{(0,0)} + O(h_t^3)]. \label{Err-yf-2-1}
\end{aligned}
\end{equation}
Notice that the above result implies the result $\widetilde{R}^i_{yf}= O(h_t^{2})$.
\end{proof}

\begin{lemma}\label{parametersofZ}
Suppose the assumptions (i)-(iii) hold. Furthermore let
$f(t,Y_t,Z_t)\in C_b^3$.
Then, for $i=0,1,\cdots,N-1$,
$$ R^i_{z}= O(h_t^{3}),\quad R^i_{z,3}= O(h_t^{3}).$$
\end{lemma}

\begin{proof}
By It\^{o}-Taylor expansions (\ref{Ito-Taylor-expansions}) and (\ref{Ito-Taylor-expansions-dw}), we have
\begin{equation}
\begin{aligned}
\mathbb{E}_{t_i}^x[Y_{t_{i+1}}\Delta W_{i,i+1}^{\top}]
= & h_tu_{t_i}^{(1)}+h_t^2u_{t_i}^{(1,0)}+\frac{h_t^{3}}{2}u_{t_i}^{(1,0,0)}+O(h_t^{4}),\\
\mathbb{E}_{t_i}^x[Z_{t_{i+1}}]
= & u_{t_i}^{(1)}+ h_tu_{t_i}^{(1,0)}+ \frac{h_t^2}{2}u_{t_i}^{(1,0,0)}+O(h_t^{3}).
\end{aligned}
\label{IT-1}
\end{equation}
By (\ref{2-2}), (\ref{2-3}) and It\^{o}-Taylor expansion (\ref{Ito-Taylor-expansions-dw}), we get
\begin{equation}
\begin{aligned}
\mathbb{E}_{t_i}^x[f_{t_{i + 1 - \alpha}} \Delta W_{i,i + 1 - \alpha}^\top ] &= -\mathbb{E}_{t_i}^x[u_{t_{i + 1 - \alpha}}^{(0)} \Delta W_{i,i + 1 - \alpha}^\top ] \\
&= - (1 - \alpha) h_t u_{t_i}^{(1,0)} - (1 - \alpha)^2 h_t^2 u_{t_i}^{(1,0,0)} - O(h_t^3), \\
\mathbb{E}_{t_i}^x[f_{t_{i + 1}} \Delta W_{i,i + 1}^\top ] &= -\mathbb{E}_{t_i}^x[u^{(0)}_{t_{i+1}} \Delta W_{i,i+1}^\top] = - h_t u_{t_i}^{(1,0)} - h_t^2 u_{t_i}^{(1,0,0)} - O(h_t^3).
\end{aligned}
\label{IT-2}
\end{equation}
Substituting (\ref{IT-1}), (\ref{IT-2}) into (\ref{FBSDE-z-t-i}), we obtain
\begin{equation}
\begin{aligned}
\frac{2}{h_t}R^i_z
= (\alpha  - \frac{1}{2})h_t^2u_{{t_i}}^{(1,0,0)} + O(h_t^3).
\end{aligned}
\label{Err-Z}
\end{equation}
From (\ref{Err-Z}), we obtain the result $R^i_{z}= O(h_t^{3})$.
Similarly, from (\ref{E-BSDE-t-i-Euler-1}), It\^{o}-Taylor expansions (\ref{Ito-Taylor-expansions}) and (\ref{Ito-Taylor-expansions-dw}), we have
\begin{align}
R^i_{z,3}
=&\alpha {h_t}{Z_{{t_{i + 1 - \alpha }}}} -\mathbb{E}_{t_{i + 1 - \alpha}}^x[{Y_{{t_{i + 1}}}}\Delta W_{i + 1 - \alpha ,i + 1}^ \top  + \alpha {h_t}{f_{{t_{i + 1}}}}\Delta W_{i + 1 - \alpha ,i + 1}^ \top]\nonumber\\
=&\mathbb{E}_{{t_{i + 1 - \alpha }}}^x\left[ {\frac{{{\alpha ^3}h_t^3}}{2}u_{{t_{i + 1 - \alpha }}}^{(1,0,0)} + O(h_t^4)} \right].\label{Err-Z-1}
\end{align}
From (\ref{Err-Z-1}), we deduce the result $R^i_{z,3}= O(h_t^{3})$.
\end{proof}

\begin{lemma} \label{parametersofYZ}
Suppose the assumptions (i)-(iii) hold. Furthermore let
$f(t,Y_t,Z_t)\in C_b^3$.
Then, for $i=0,1,\cdots,N-1$,
$$ R^i_{y}= O(h_t^{3}).$$
\end{lemma}
\begin{proof}
From the definition of $\widehat{R}^i_{yf}$, we know
\begin{equation}
\begin{aligned}
\widehat{R}^i_{yf}=& h_t\mathbb{E}_{t_i}^x[f_{t_{i + 1 - \alpha }}-\widetilde{f}_{t_{i + 1 - \alpha }}]\\
\le & h_t|\mathbb{E}_{t_i}^x[f_{t_{i + 1 - \alpha }}-\widetilde{f}_{t_{i + 1 - \alpha }}]|.\label{R-i-yf-1}
\end{aligned}
\end{equation}
By the Lipschitz condition and the definitions of $\widetilde{Y}_{t_{i + 1 - \alpha }} $ and $\widetilde{Z}_{t_{i + 1 - \alpha }}$, we can rewrite (\ref{R-i-yf-1}) as
\begin{equation}
\begin{aligned}
\widehat{R}^i_{yf}
\le& h_tL\mathbb{E}_{t_i}^x[|Y_{t_{i + 1 - \alpha }}-\widetilde{Y}_{t_{i + 1 - \alpha }}| + |Z_{t_{i + 1 - \alpha }}-\widetilde{Z}_{t_{i + 1 - \alpha }}|] \\
=& {h_t}L\mathbb{E}_{{t_i}}^x[|{Y_{{t_{i + 1 - \alpha }}}} - \mathbb{E}_{{t_{i + 1 - \alpha }}}^x[{Y_{{t_{i + 1}}}} + \alpha {h_t}{f_{{t_{i + 1}}}}]|\\
&+ |{Z_{{t_{i + 1 - \alpha }}}} -\mathbb{E} _{{t_{i + 1 - \alpha }}}^x[\frac{1}{{\alpha {h_t}}}{Y_{{t_{i + 1}}}}\Delta W_{i + 1 - \alpha ,i + 1}^ \top  + {f_{{t_{i + 1}}}}\Delta W_{i + 1 - \alpha ,i + 1}^ \top ]|]\\
=&h_tL\mathbb{E}_{t_i}^x[|\widetilde{R}^i_{yf}| + |\frac{1}{\alpha h_t}R^i_{z,3}|].\nonumber
\end{aligned}
\end{equation}
From Lemmas \ref{parametersofY} and \ref{parametersofZ}, we have
\begin{equation}
\begin{aligned}
\widehat{R}^i_{yf}=O(h_t^3).\label{R-i-yf-2}
\end{aligned}
\end{equation}
By the definition of $R^i_{y}$, (\ref{R-i-yf-2}) and Lemma \ref{parametersofY}, we obtain
\begin{equation}
\begin{array}
[c]{rl}%
&R^i_y=R^i_{yf}+\widehat{R}^i_{yf}= O(h_t^3) .
\end{array}
\nonumber
\end{equation}
The proof is completed.
\end{proof}

In what follows, we focus on analyzing the convergence property of the scheme (\ref{NumSch}). 
\begin{theorem}\label{discretization-error}
Assuming that conditions (i)-(iii) and \(f(t,Y_t,Z_t)\in C_b^3\) are met, we further require that the initial values satisfy
$\max\{\mathbb{E}[|Y_{t_{N}}-Y_{N}^{\pi}|],\mathbb{E}[|Z_{t_{N}}-Z_{N}^{\pi}|]\}=O(h^{2}_t)$.
Then we have, for sufficiently small time step $h_t$
\begin{equation}
\mathbb{E}[\sup\limits_{0\le i\le N}\left|\Delta Y_{i}\right|^2]
+ \sum_{\ell=i}^{N-1}\mathbb{E}[h_t\left|\Delta Z_{\ell}\right|^2]
 \le Ch_t^{4},\label{3-9}
\end{equation}
where 
$C$ denotes a positive constant.
\end{theorem}
\begin{proof}
Let $\Delta Y_{i}=Y_{t_i} -Y_{i}^{\pi},\Delta Z_{i}=Z_{t_i}-Z_{i}^{\pi},$
$\Delta {Y_{i + 1 - \alpha }} = {Y_{{t_{i + 1 - \alpha }}}} - Y_{i + 1 - \alpha }^\pi, \Delta {Z_{i + 1 - \alpha }} = {Z_{{t_{i + 1 - \alpha }}}} - Z_{i + 1 - \alpha }^\pi$.
From (\ref{FBSDE-z-t-i}) and (\ref{GBSDE-t-i-z-pi}),
we have
\begin{align}
\left|\Delta Z_{i}\right|
= & \left| \mathbb{E}_{t_i}^x \left[
\Delta Y_{i + 1} \frac{2 \Delta W_{i,i + 1}^\top}{h_t}
+ \frac{1}{\alpha} ( f_{i + 1 - \alpha} - \widetilde{f}_{i + 1 - \alpha}^\pi ) \Delta W_{i,i + 1 - \alpha}^\top \right. \right. \nonumber \\
& \left. \left.
+ \frac{2\alpha - 1}{\alpha} \left( f_{t_{i + 1}} - f_{i + 1}^\pi \right) \Delta W_{i,i + 1}^\top
- \Delta Z_{i + 1} \right]
+ \frac{2}{h_t} R_z^i \right|.
\label{3-10}
\end{align}
The expression in (\ref{3-10}) can be reformulated by applying the Lipschitz condition, as follows:
\begin{align}
\left|\Delta Z_{i}\right|
\le & \left|\mathbb{E}_{t_i}^x\left[ \Delta Y_{i + 1} \frac{2 \Delta W_{i,i + 1}^\top}{h_t} \right] \right|
+ \frac{L}{\alpha} \mathbb{E}_{t_i}^x\left[ \left( |\Delta Y_{i + 1 - \alpha}| + |\Delta Z_{i + 1 - \alpha}| \right) |\Delta W_{i,i + 1 - \alpha}^\top| \right] \nonumber \\
& + \left| \frac{2\alpha - 1}{\alpha} \right| L \mathbb{E}_{t_i}^x\left[\left( |\Delta Y_{i + 1}| + |\Delta Z_{i + 1}| \right) |\Delta W_{i,i + 1 - \alpha}^\top|\right]
 + \mathbb{E}_{t_i}^x\left[ |\Delta Z_{i + 1}| \right]
 + \left| \frac{2}{h_t} R_z^i \right|.
\label{3-11}
\end{align}
From the Cauchy-Schwarz inequality, the inequality (\ref{3-11}) is represented as
\begin{align}
\left| \Delta Z_i \right|
\le & \left| \frac{2}{h_t} \sqrt{h_t \mathbb{E}_{t_i}^x[|\Delta Y_{i + 1}|^2]} \right| + \frac{L}{\alpha} \left( \sqrt{(1 - \alpha) h_t \mathbb{E}_{t_i}^x[|\Delta Y_{i + 1 - \alpha}|^2]} + \sqrt{(1 - \alpha) h_t \mathbb{E}_{t_i}^x[|\Delta Z_{i + 1 - \alpha}|^2]} \right) \nonumber \\
& + \left| \frac{2\alpha - 1}{\alpha} \right| L \left( \sqrt{h_t \mathbb{E}_{t_i}^x[|\Delta Y_{i + 1}|^2]} + \sqrt{h_t \mathbb{E}_{t_i}^x[|\Delta Z_{i + 1}|^2]} \right)
 + \mathbb{E}_{t_i}^x[|\Delta Z_{i + 1}|] + \left| \frac{2}{h_t} R_z^i \right|.
\label{3-12}
\end{align}
From (\ref{FBSDE-err-y-2}) and (\ref{E-GBSDE-t-i-pi}),we have
\begin{equation}
\begin{aligned}
|\Delta Y_{i + 1 - \alpha}|^2
=&\left|\mathbb{E}_{t_{i + 1 - \alpha}}^x\left[\Delta Y_{i+1}+ \alpha h_t(f_{t_{i+1}}-f^\pi_{i+1})\right]+ \widetilde{R}^i_{yf}\right|^2\\
\le&\left|\mathbb{E}_{t_{i + 1 - \alpha}}^x\left[\Delta Y_{i+1}+\alpha h_tL(|\Delta Y_{i+1}| + |\Delta Z_{i+1}|)\right]+ \widetilde{R}^i_{yf}\right|^2.
\end{aligned}
\label{Err-Y��t-1/2}
\end{equation}
From (\ref{E-BSDE-t-i-Euler-1}) and (\ref{E-GBSDE-t-i-pi-1}),we have
\begin{equation}
\begin{aligned}
|\Delta Z_{i + 1 - \alpha}|
=&\left|\mathbb{E}_{t_{i + 1 - \alpha}}^x\left[\frac{1}{\alpha h_t}\Delta Y_{i+1}\Delta W^\top_{i + 1 - \alpha,i+1}
+(f_{t_{i+1}}-f_{i+1}^\pi)\Delta W^\top_{i + 1 - \alpha,i+1}\right] +R^i_{z,3}\right|^2\\
\le& \left|\mathbb{E}_{t_{i + 1 - \alpha}}^x\left[\frac{1}{\alpha h_t}\Delta Y_{i+1}\Delta W^\top_{i + 1 - \alpha,i+1}
+L(|\Delta {Y_{i + 1}}| + |\Delta {Z_{i + 1}}|)\Delta W^\top_{i + 1 - \alpha,i+1}\right] +R^i_{z,3}\right|^2\\
\le&2\mathbb{E}_{t_{i + 1 - \alpha}}^x\left[(\frac{2}{\alpha h_t}+2\alpha h_tL^2+4L)|\Delta Y_{i+1}|^2+2\alpha h_tL^2|\Delta Z_{i+1}|^2\right] +2\left|R^i_{z,3}\right|^2.
\end{aligned}
\label{Err-Z��t-1/2}
\end{equation}
Plugging (\ref{Err-Y��t-1/2}) and (\ref{Err-Z��t-1/2}) into (\ref{3-12}), we obtain
\begin{align}
\left| \Delta Z_i \right|
\le & \left| \frac{2}{h_t} \sqrt{{h_t} \mathbb{E}_{t_i}^x[|\Delta Y_{i + 1}|^2]} \right|
+ \frac{L}{\alpha} \sqrt{2(1 - \alpha) h_t \mathbb{E}_{t_i}^x \left[ \left( (1 + \alpha h_t L) |\Delta Y_{i + 1}| + \alpha h_tL|\Delta Z_{i + 1}| \right)^2 + (\widetilde{R}_{yf}^i)^2 \right]} \nonumber\\
& + \frac{L}{\alpha} \sqrt{2(1 - \alpha) h_t \mathbb{E}_{t_i}^x \left[ \left( \frac{2}{\alpha h_t} + 2\alpha h_t L^2 + 4L \right) |\Delta Y_{i + 1}|^2 + 2\alpha h_t L^2 |\Delta Z_{i + 1}|^2 + |R_{z,3}^i|^2 \right]} \nonumber\\
& + \left| \frac{2\alpha - 1}{\alpha} \right| L \left( \sqrt{h_t \mathbb{E}_{t_i}^x[|\Delta Y_{i + 1}|^2]} + \sqrt{h_t \mathbb{E}_{t_i}^x[|\Delta Z_{i + 1}|^2]} \right)  + \mathbb{E}_{t_i}^x[|\Delta Z_{i + 1}|] + \left| \frac{2}{h_t} R_z^i \right|.
\label{3-12-11}
\end{align}
Squaring (\ref{3-12-11}), we deduce
\begin{align}
\left|\Delta Z_{i}\right|^2
\le & \, \left(\frac{20}{h_t} + \frac{10(1 - \alpha)h_t L^2}{\alpha^2}
\left(4 + 4\alpha^2 h_t L^2 + 8\alpha h_t L + \frac{2}{\alpha h_t} + 2\alpha h_t L^2 + 4L\right) \right. \nonumber \\
& \left. + \frac{10{(2\alpha - 1)}^2 h_t L^2}{\alpha^2}\right)\mathbb{E}_{t_i}^x[|\Delta Y_{i+1}|^2] \nonumber \\
& + \left(\frac{10(1 - \alpha)h_t L^2}{\alpha^2}(4\alpha^2 h_t^2 L^2 + 2\alpha h_t L^2)
+ \frac{10{(2\alpha - 1)}^2 h_t L^2}{\alpha^2} + 5 \right)\mathbb{E}_{t_i}^x[|\Delta Z_{i+1}|^2] \nonumber \\
& + \frac{20(1 - \alpha)h_t L^2}{\alpha^2}\left(|\widetilde{R}^i_{yf}|^2 + |R^i_{z,3}|^2\right)
+ 5\left|\frac{2}{h_t} R^i_z\right|^2.
\label{3-12-1}
\end{align}
Next, we deal with the term $\left| Y_{t_i} - Y_i^{\pi} \right| $.
From  (\ref{FBSDE-err-y-3}) and  (\ref{E-GBSDE-t-i-M-Euler-pi}), we obtain
\begin{align}
\left|\Delta Y_{i}\right|
\le & \left| \mathbb{E}_{t_i}^x[\Delta Y_{i + 1}] \right|
+ \left| \frac{h_t}{2\alpha} \mathbb{E}_{t_i}^x[\widetilde{f}_{t_{i + 1 - \alpha}} - \widetilde{f}_{i + 1 - \alpha}^\pi] \right|
+ \left| \frac{(2\alpha - 1) h_t}{2\alpha} \mathbb{E}_{t_i}^x[f_{t_{i + 1}} - f_{i + 1}^\pi] \right|
+ \left| R_y^i \right| \nonumber \\
\le & \left| \mathbb{E}_{t_i}^x[\Delta Y_{i+1}] \right|
+ \frac{h_t}{2\alpha} L \mathbb{E}_{t_i}^x \left[ |\widetilde{Y}_{t_{i + 1 - \alpha}} - Y^\pi_{i + 1 - \alpha}| + |\widetilde{Z}_{t_{i + 1 - \alpha}} - Z^\pi_{i + 1 - \alpha}| \right] \nonumber \\
& + \left|\frac{2\alpha - 1}{2\alpha}\right| h_t L\mathbb{E}_{t_i}^x \left[ \left| \Delta Y_{i + 1} \right| + \left| \Delta Z_{i + 1} \right| \right]
+ \left| R_y^i \right|.
\label{3-14}
\end{align}
By the definition of $\widetilde{Y}_{t_{i + 1 - \alpha}}$ and (\ref{E-GBSDE-t-i-pi}), we get
\begin{align}
\left| \widetilde{Y}_{t_{i + 1 - \alpha}} - Y_{i + 1 - \alpha}^\pi \right|
=& \left| \mathbb{E}_{t_{i + 1 - \alpha}}^x \left[ Y_{t_{i + 1}} + \alpha h_t f_{t_{i + 1}} \right]
- \mathbb{E}_{t_{i + 1 - \alpha}}^x \left[ Y_{i + 1}^\pi + \alpha h_t f_{i + 1}^\pi \right] \right| \nonumber \\
\le& (1 + \alpha h_t L) \mathbb{E}_{t_i}^x \left[ |\Delta Y_{i + 1}| \right]
+ \alpha h_t L \mathbb{E}_{t_i}^x \left[ |\Delta Z_{i + 1}| \right].
\label{3-14-1}
\end{align}
By the definition of $\widetilde{Z}_{t_{i + 1 - \alpha}}$ and (\ref{E-GBSDE-t-i-pi-1}), we deduce
\begin{align}
\left| \widetilde{Z}_{t_{i + 1 - \alpha}} - Z_{i + 1 - \alpha}^\pi \right|
\le & \left| \mathbb{E}_{t_{i + 1 - \alpha}}^x \left[ \frac{1}{\alpha h_t} \Delta Y_{i + 1} \Delta W_{i + 1 - \alpha, i + 1}^\top + L (|\Delta Y_{i + 1}| + |\Delta Z_{i + 1}|) \Delta W_{i + 1 - \alpha, i + 1}^\top \right] \right|.
\label{3-14-2}
\end{align}
Substituting (\ref{3-14-1}) and (\ref{3-14-2}) into (\ref{3-14}), we obtain
\begin{align}
\left|\Delta Y_{i}\right|
\le& \left(1 + \frac{h_t L}{2\alpha} (1 + \alpha h_t L) + \left| \frac{2\alpha - 1}{\alpha} \right| h_t L\right) \mathbb{E}_{t_i}^x \left[\left|\Delta Y_{i+1}\right|\right] \nonumber \\
& + \left( \frac{h_t^2 L^2}{\alpha^2} + \left| \frac{2\alpha - 1}{\alpha} \right| h_t L \right) \mathbb{E}_{t_i}^x \left[\left|\Delta Z_{i+1}\right|\right]
+ \left( \frac{L}{2\alpha^2} + \frac{h_t L^2}{2\alpha} \right) \mathbb{E}_{t_i}^x \left[\left|\Delta Y_{i+1} \Delta W^\top_{i + 1 - \alpha, i + 1}\right|\right] \nonumber \\
& + \frac{h_t L^2}{2\alpha} \mathbb{E}_{t_i}^x \left[\left|\Delta Z_{i+1} \Delta W^\top_{i + 1 - \alpha, i + 1}\right|\right]
+ \left|R^i_y\right|.
\label{3-14-3}
\end{align}
Squaring (\ref{3-14-3}), we deduce
\begin{align}
\left|\Delta Y_{i}\right|^2
\le& \, 5 \left( 4 + \frac{h_t^2 L^2}{\alpha^2} + h_t^4 L^4 + \frac{4(2\alpha - 1)^2}{\alpha^2} h_t^2 L^2 + \frac{h_t L^2}{2 \alpha^3} + \frac{h_t^3 L^4}{2 \alpha} \right) \mathbb{E}_{t_i}^x \left[|\Delta Y_{i + 1}|^2\right] \nonumber \\
&+ 5 \left( \frac{h_t^4 L^4}{2} + \frac{2(2\alpha - 1)^2}{\alpha^2} h_t^2 L^2 + \frac{h_t^3 L^4}{4 \alpha} \right) \mathbb{E}_{t_i}^x \left[|\Delta Z_{i + 1}|^2\right]
+ 5 \left|R_y^i\right|^2.
\label{3-14-4}
\end{align}
Adding (\ref{3-12-1}) to (\ref{3-14-4}), we deduce
\begin{align}
&\left|\Delta Y_{i}\right|^2 + h_t\left|\Delta Z_{i}\right|^2 \nonumber\\
\le& \, 5 \bigg( 8 + \frac{4(1 - \alpha)}{\alpha^3} L^2 h_t + \frac{6(2\alpha - 1)^2 + 1 + 8(1 - \alpha) + 8L(1 - \alpha)}{\alpha^2} L^2 h_t^2 \nonumber\\
& \quad + \frac{L + 32(1 - \alpha) + 8L(1 - \alpha)}{2\alpha} L^3 h_t^3 + (1 + 8(1 - \alpha)) L^4 h_t^4 \bigg) \mathbb{E}_{t_i}^x \left[|\Delta Y_{i+1}|^2\right] \nonumber\\
& + 5 \left( h_t + \frac{4(2\alpha - 1)^2}{\alpha^2} L^2 h_t^2 + \frac{16(1 - \alpha) + 1}{4\alpha} L^4 h_t^3 + \frac{16(1 - \alpha) + 1}{2} L^4 h_t^4 \right) \mathbb{E}_{t_i}^x \left[|\Delta Z_{i+1}|^2 \right] \nonumber\\
& + 5 \left|R^i_y\right|^2 + \frac{20(1 - \alpha) h_t^2 L^2}{\alpha^2} \left(|\widetilde{R}^i_{yf}|^2 + |R^i_{z,3}|^2\right) + \frac{20}{h_t} \left|R^i_z\right|^2 \nonumber\\
=& \, 40(1 + C_1 h_t) \mathbb{E}_{t_i}^x \left[|\Delta Y_{i+1}|^2\right] + C_2 h_t \mathbb{E}_{t_i}^x \left[|\Delta Z_{i+1}|^2 \right] \nonumber\\
& + 5 \left|R^i_y\right|^2 + \frac{20(1 - \alpha) h_t^2 L^2}{\alpha^2} \left(|\widetilde{R}^i_{yf}|^2 + |R^i_{z,3}|^2\right) + \frac{20}{h_t} \left|R^i_z\right|^2,
\label{3-14-5}
\end{align}
where $C_1 = \frac{1}{8}(\frac{{4(1 - \alpha )}}{{{\alpha ^3}}}{L^2} + \frac{{6{{(2\alpha  - 1)}^2} + 1 + 8(1 - \alpha ) + 8L(1 - \alpha )}}{{{\alpha ^2}}}{L^2}{h_t} + \frac{{L + 32(1 - \alpha ) + 8L(1 - \alpha )}}{{2\alpha }}{L^3}{h_t}^2 + (1 + 8(1 - \alpha )){L^4}{h_t}^3)$, $C_2 =5(1 + \frac{{4{{(2\alpha  - 1)}^2}}}{{{\alpha ^2}}}{L^2}{h_t} + \frac{{16(1 - \alpha ) + 1}}{{4\alpha }}{L^4}{h_t}^2 + \frac{{16(1 - \alpha ) + 1}}{2}{L^4}{h_t}^3).$
There exist two positive constants $C_0$ and $C_3$ such that
\begin{align}
&C_0\left|\Delta Y_{i}\right|^2
+ h_t\left|\Delta Z_{i}\right|^2\nonumber\\
\le&
(1+C_3h_t)\left(C_0\mathbb{E}_{t_i}^x[|\Delta Y_{i+1}^{\pi}|^2]
+C_2h_t\mathbb{E}_{t_i}^x[|\Delta Z_{i+1}^{\pi}|^2 ]\right)\nonumber\\
& +C_3\left|R^i_y\right|^2 +C_3h_t^2\left((\widetilde{R}^i_{yf})^2+|R^i_{z,3}|^2\right)
+\frac{C_3}{h_t}\left|R^i_z\right|^2.
\label{3-14-5-10}
\end{align}
From (\ref{3-14-5-10}) and the tower property of the conditional expectations, we deduce
\begin{align}
&C_0\mathbb{E}[\left|\Delta Y_{i}\right|^2]
+ (1-C_2)h_t\sum_{\ell=i}^{N-1}(1+C_3h_t)^{\ell-i}\mathbb{E}[\left|\Delta Z_{\ell}\right|^2]\nonumber\\
\le&
(1+C_3h_t)^{N-i}\left(C_0\mathbb{E}[|\Delta Y_{N}|^2]
+C_2h_t\mathbb{E}[|\Delta Z_{N}|^2]\right)\nonumber\\
& +\sum_{\ell=i}^{N-1}(1+C_3h_t)^{\ell-i}\mathbb{E}\left[
C_3\left|R^\ell_y\right|^2 +C_3h_t^2\left((\widetilde{R}^\ell_{yf})^2+|R^\ell_{z,3}|^2\right)
+\frac{C_3}{h_t}\left|R^\ell_z\right|^2\right].
\label{3-14-5-11}
\end{align}
Furthermore, we have
\begin{align}
\mathbb{E}[\left|\Delta Y_{i}\right|^2]
+ h_t\sum_{\ell=i}^{N-1}\mathbb{E}[\left|\Delta Z_{\ell}\right|^2]
\le&
C\left(\mathbb{E}[|\Delta Y_{N}|^2]
+h_t\mathbb{E}[|\Delta Z_{N}|^2 ]\right)\nonumber\\
& +\sum_{\ell=i}^{N-1}C\mathbb{E}\left[
\left|R^\ell_y\right|^2 +h_t^2\left((\widetilde{R}^\ell_{yf})^2+|R^\ell_{z,3}|^2\right)
+\frac{1}{h_t}\left|R^\ell_z\right|^2\right].
\label{3-14-5-12}
\end{align}
From Lemmas \ref{parametersofY}, \ref{parametersofZ} and \ref{parametersofYZ}, we have
\begin{align}
\mathbb{E}[\left|R^i_y\right|^2 +h_t^2\left(|\widetilde{R}^i_{yf}|^2+|R^i_{z,3}|^2\right)
+\frac{1}{h_t}\left|R^i_z\right|^2]
=& O(h_t^5).
\label{3-14-5-1}
\end{align}
From $\max\{\mathbb{E}[|Y_{t_{N}}-Y_{N}^{\pi}|],\mathbb{E}[|Z_{t_{N}}-Z_{N}^{\pi}|]\}=O(h_t^{2})$ and (\ref{3-14-5-1}), we have the following estimates
\begin{align}
\mathbb{E}[\left|\Delta Y_{i}\right|^2]
+ \sum_{\ell=i}^{N-1}\mathbb{E}[h_t\left|\Delta Z_{\ell}\right|^2]
 \le Ch_t^{4}.
\nonumber
\end{align}
The proof is completed.
\end{proof}

\begin{flushleft}
\refstepcounter{remark}
\textbf{\theremark} 
Through the proof provided in Theorem \ref{discretization-error}, we have successfully demonstrated that the proposed one step scheme achieves a second-order convergence rate. This result is further confirmed by the numerical examples constructed in the following section.
\end{flushleft}
\begin{flushleft}
\refstepcounter{remark}
\textbf{\theremark} 
 This paper differs from \cite{ZWCLPS06} and \cite{CDMK14} in that our one step scheme is fully explicit for both $Y$ and $Z$, whereas the schemes in \cite{ZWCLPS06} and \cite{CDMK14} are explicit for $Y$ but implicit for $Z$. Moreover, when $\alpha = 1$, our method reduces to the Crank Nicolson method, as discussed in \cite{CDMK14, ZWCLPS06}, demonstrating that the Crank Nicolson method is a special case within our framework.
\end{flushleft}

\section{Numerical experiments}


In this section, we illustrate the accuracy and effectiveness of the aforementioned scheme through numerical experiments. The time interval $[0, T]$ is uniformly divided into $N$ parts, giving a time step size $h_t = \frac{T}{N}$, where $T$ represents the terminal time. The time step sizes adopted in our experiments are $N = \frac{1}{2^m}$ ($m = 3, \ldots, 7$). To compute the conditional expectation \( \mathbb{E}_{t_n}^x[\cdot] \) within the proposed scheme, we utilize the Gauss-Hermite quadrature method, together with spatial interpolation. By choosing a sufficiently large number of Gauss-Hermite quadrature points \( K \), the local truncation error becomes insignificant; in the numerical examples provided, we set \( K = 12 \).
We apply cubic spline interpolation for the spatial interpolation process.

In the tables below, the errors $\text{err}_y := |Y_0 - Y_0^{\pi}|$ and $\text{err}_z := |Z_0 - Z_0^{\pi}|$ denote the differences between the exact solution \((Y_t, Z_t)\) of equation (1.1) at \(t = 0\) and the corresponding numerical results \((Y_0^{\pi}, Z_0^{\pi})\) obtained from the scheme.
 Let CR denote the convergence rate with respect to the time step size \(h_t\), which can be calculated using the least squares regression method. All the numerical tests are implemented in Matlab R2023b on a desktop computer with Intel Core i5-12600KF 10-Core Processor (4.9GHz) and 32 GB DDR4 RAM (3600MHz). The solution process is presented in the following algorithm.
\begin{center}
\begin{minipage}{14cm} 
\captionsetup{labelformat=empty} 
\begin{algorithm}[H]
    \SetAlgoNoLine 
	\SetAlgoLined
	\caption{The algorithm for solving FBSDEs based on Scheme (\ref{NumSch}).} 
	\KwIn{$K, \alpha, N, Y_N^\pi, Z_N^\pi$}
	
	\For{$i = N-1$ \KwTo $0$}{
		Compute the interpolation points at time level \( t_{n+1} \)\
		
		\For{$k = 1$ \KwTo $K$}{
			$Y^\pi_{i+1-\alpha} = \mathbb{E}_{t_{i+1-\alpha}}^x\left[Y^\pi_{i+1}+\alpha h_t f^\pi_{i+1}\right]$ \\
            $Z^\pi_{i+1-\alpha} = \mathbb{E}_{t_{i+1-\alpha}}^x\left[\frac{1}{\alpha h_t}Y^\pi_{i+1}\Delta W^\top_{i+1-\alpha,i+1}+f_{i+1}^\pi\Delta W^\top_{i+1-\alpha,i+1}\right]$ \\
		}
		
		\For{$k = 1$ \KwTo $K$}{
			$Y_i^\pi = \mathbb{E}_{t_i}^x\left[Y^\pi_{i+1}+ \frac{h_t}{2\alpha}\widetilde{f}^\pi_{i+1-\alpha }+ h_t\left(1-\frac{1}{2\alpha}\right)f^\pi_{i+1}\right]$ \\
            $Z^\pi_i = \mathbb{E}_{t_i}^x\left[\frac{2}{h_t}Y^\pi_{i+1}\Delta W_{i,i+1}^{\top}
            +\frac{1}{\alpha}\widetilde{f}^\pi_{i+1-\alpha }\Delta W_{i,i+1-\alpha}^{\top}+ \frac{2\alpha-1}{\alpha}f^\pi_{i+1}\Delta W_{i,i+1}^{\top}
            -Z^\pi_{i+1}\right]$ \\
		}
	}
	
	\KwOut{$Y_0^\pi$, $Z^\pi_0$}
\end{algorithm}
\end{minipage}
\end{center}

\textbf{Example 1.} Consider the BSDE as below:
\begin{equation}
\begin{cases}
-dY_t = (-Y_t^3 + 2.5Y_t^2 - 1.5Y_t)dt - Z_t dW_t, & \text{for } t \in [0, T], \\
Y_T = \dfrac{\exp(W_T + T)}{\exp(W_T + T) + 1},
\end{cases}
\label{BSDEdetail}
\end{equation}
The solution to the  BSDE (\ref{BSDEdetail}) is

\[
Y_t = \frac{\exp(W_t + t)}{\exp(W_t + t) + 1}, \quad Z_t = \frac{\exp(W_t + t)}{(\exp(W_t + t) + 1)^2}.
\]


In this example, we set \( T = 1 \) and the exact solution of the equation (\ref{BSDEdetail}) at \( t = 0 \) is \( (Y_0, Z_0) = (0.5, 0.25) \). Table \ref{Err-Y-Z-ex1-1} presents the absolute errors, \(\text{err}_y\) and \(\text{err}_z\), along with their convergence rates, for various values of \(\alpha\) using the described method. The run time of the scheme for Example 1 is provided in Table \ref{Err-Y-Z-ex1-2}. Figure \ref{fig:Example1} shows the log-log plots of err\(_y\) and err\(_z\) against time interval \(h_t\).

\begin{table*}[ht]
\caption{Error values and convergence rates for the scheme in Example 1.}
\begin{center}
  \setlength{\tabcolsep}{4pt} 
  \renewcommand{\arraystretch}{1.1} 
  \small
  \begin{tabular}{@{}c|c|c|c|c|c|c|c|c@{}}
    \Xhline{1pt}
    \multirow{2}{*}{$N$} & \multicolumn{2}{c|}{$\alpha = \frac{1}{4}$} & \multicolumn{2}{c|}{$\alpha = \frac{1}{2}$} & \multicolumn{2}{c|}{$\alpha = \frac{3}{4}$} & \multicolumn{2}{c@{}}{$\alpha = 1$} \\ \cline{2-9}
    & $|Y_0 - Y^0|$ & $|Z_0 - Z^0|$ & $|Y_0 - Y^0|$ & $|Z_0 - Z^0|$ & $|Y_0 - Y^0|$ & $|Z_0 - Z^0|$ & $|Y_0 - Y^0|$ & $|Z_0 - Z^0|$ \\ \Xhline{1pt}
    8 & 1.3590E-04 & 2.4418E-05 & 1.2017E-04 & 1.0366E-04 & 1.0258E-04 & 1.9060E-04 & 8.2793E-05 & 2.8535E-04 \\
    16 & 3.4884E-05 & 4.9078E-06 & 3.0986E-05 & 2.6713E-05 & 2.6797E-05 & 4.9519E-05 & 2.2350E-05 & 7.3383E-05 \\
    32 & 8.8581E-06 & 1.1203E-06 & 7.8802E-06 & 6.8085E-06 & 6.8659E-06 & 1.2627E-05 & 5.8187E-06 & 1.8580E-05 \\
    64 & 2.2179E-06 & 2.3749E-07 & 1.9729E-06 & 1.6882E-06 & 1.7234E-06 & 3.1556E-06 & 1.4697E-06 & 4.6399E-06 \\
   128 & 5.5778E-07 & 5.2154E-08 & 4.9651E-07 & 4.1841E-07 & 4.3466E-07 & 7.8672E-07 & 3.7230E-07 & 1.1572E-06 \\
    \Xhline{1pt}
    CR    & 1.9833  & 2.2111  & 1.9811  & 1.9889  & 1.9724  & 1.9813  & 1.9520  & 1.9875  \\
    \Xhline{1pt}
  \end{tabular}
\end{center}
\label{Err-Y-Z-ex1-1}
\end{table*}

\begin{table}[!ht]
\vspace{-1.5em}
\caption{Run time({s}) of the scheme for Example 1.}
\centering
\setlength{\tabcolsep}{4pt} 
\renewcommand{\arraystretch}{1.1} 
\small
\begin{tabular}{c|ccccc}
\toprule
\diagbox[width=6em,height=2.5em]{$\alpha$}{N} & 8 & 16 & 32 & 64 & 128 \\
\midrule
$1/4$ & 4.5610  & 9.0057  & 18.0326  & 36.4851  & 74.9443  \\
$1/2$ & 4.5909  & 9.0988  & 18.1107  & 36.1546  & 75.3838  \\
$3/4$ & 4.5842  & 9.0513  & 18.6429  & 38.0595  & 74.9173  \\
$1$   & 4.5079  & 8.9543  & 17.9163  & 36.0204  & 74.6518  \\
\bottomrule
\end{tabular}
\label{Err-Y-Z-ex1-2}
\end{table}
\clearpage
\begin{figure}[!ht]
\begin{center}
\includegraphics[
trim=0.000000in 0.000000in 0.000000in -0.398623in, 
width=0.9\textwidth] 
{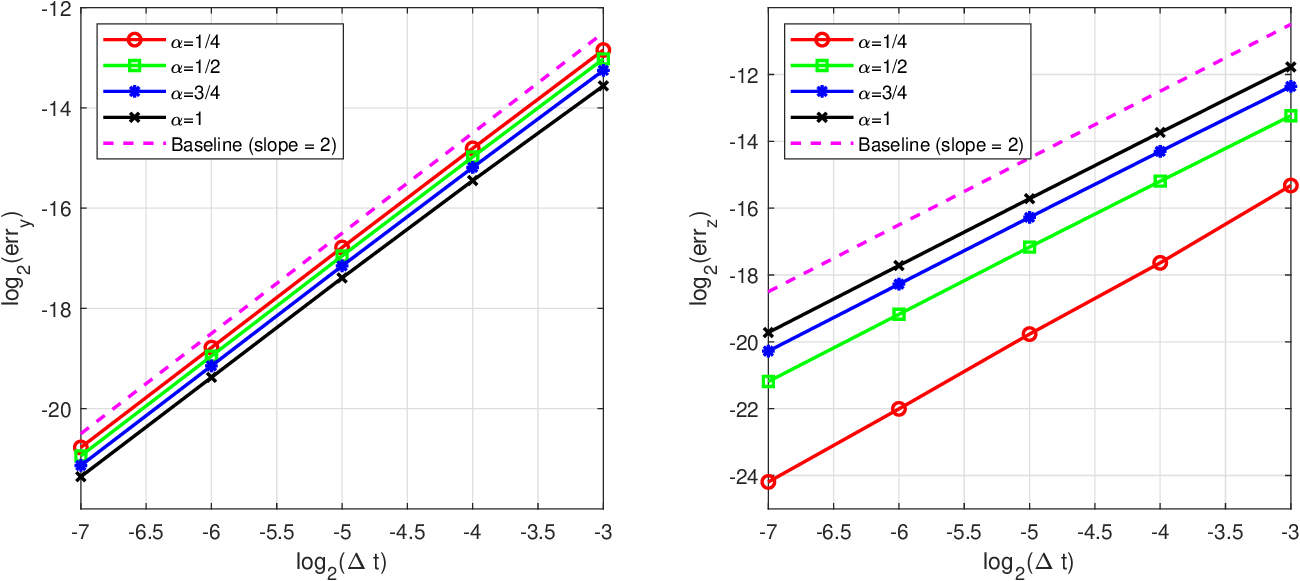} 
\caption{The plot of $\text{err}_{y}$ against time interval \(\Delta t\) in Example 1.}
\label{fig:Example1}
\end{center}
\end{figure}

\textbf{Example 2.} The FitzHugh-Nagumo (FHN) equation is widely used in biology and genetics, particularly in the mathematical modeling of electrophysiological systems and the mathematical modeling of neuronal dynamics.The following is the partial differential equation (PDE) form of a simplified case of the FHN equation (taken from \cite{LADRGSL15}):

\begin{equation}
\left\{
\begin{aligned}
&- \partial_t u - \frac{1}{2} \Delta u - (cu^3 + bu^2 - au) = 0, \\
&u(T, x) = g(x).
\end{aligned}
\right. \label{4-2}
\end{equation}
Let \( c = -1 \), \( b = 1 + a \), \( a \in \mathbb{R} \), and \( g(x) = (1 + e^x)^{-1} \), then

\begin{equation}
u(t, x) = (1 + \exp \{ x - (0.5 - a)(T - t) \})^{-1} \in C_b^\infty([0, T] \times \mathbb{R}),
\end{equation}
is the solution of the equation \eqref{4-2}. The corresponding FBSDEs is given by:

\begin{equation}
\left\{
\begin{aligned}
&X_t = x_0 + \int_0^t dW_s, \\
&Y_t = g(X_T) + \int_t^T [-Y_s^3 + (1 + a)Y_s^2 - aY_s] ds - \int_t^T Z_s dW_s.
\end{aligned}
\right. \label{4-3}
\end{equation}
From It\^o-Taylor formula, the analytical solutions of the BSDE \eqref{4-3} can be represented in the following form:

\[
Y_t = \frac{1}{1 + \exp \{ X_t - (0.5 - a)(T - t) \}}, \quad Z_t = -\frac{\exp \{ X_t - (0.5 - a)(T - t) \}}{(1 + \exp \{ X_t - (0.5 - a)(T - t) \})^2}.
\]

Let \(T=1\), \(a=-0.5\), and \(x_0=1\), then the exact solution of the equation (\ref{4-3}) at \( t = 0 \) is \( (Y_0, Z_0) = (0.5, -0.25) \). Table \ref{Err-Y-Z-ex2-1} presents the absolute errors \( \text{err}_y \) and \( \text{err}_z \) obtained using the aforementioned method for the various values of \(\alpha\), accompanied by their respective convergence rates. The run time of the scheme for Example 2 is provided in Table \ref{Err-Y-Z-ex2-1.5}. Figure \ref{fig:Example2-0.5} shows the log-log plots of err\(_y\) and err\(_z\) against time interval \(h_t\).

Let \(T=1\), \(a=-1\), and \(x_0=1.5\), then the the exact solution of the equation (\ref{4-3}) at \( t = 0 \) is \( (Y_0, Z_0) = (0.5, 0.25) \). Table \ref{Err-Y-Z-ex2-2} presents the absolute errors \( \text{err}_y \) and \( \text{err}_z \) obtained using the aforementioned method for the various values of \(\alpha\), accompanied by their respective convergence rates. The run time of the scheme for Example 2 is provided in Table \ref{Err-Y-Z-ex2-2.5}. Figure \ref{fig:Example2-2} shows the log-log plots of err\(_y\) and err\(_z\) against time interval \(h_t\).

From the errors, convergence rates and log-log plots listed in Tables 1–6 and Figure 1-3, we conclude that:

1. In the examples provided, the scheme demonstrates a convergence rate close to order two, as illustrated by Tables 1, 3, and 5. This finding aligns with the second order error estimate presented in Theorem 4.5.

2. As illustrated in Figures 1-3, the slopes of the fitted curves closely align with the baseline slope of 2, regardless of the value of $\alpha$. The specific outcomes for different $\alpha$ are influenced by the particular examples used.

3. The run time in Tables 2, 4 and 6 demonstrate that, regardless of the value of $\alpha$, the scheme exhibits approximately the same efficiency. In the examples discussed, for a fixed N, the run time for the four different values of $\alpha$ show minimal variation.

\begin{table*}[!ht]
\caption{Error values and convergence rates for the scheme in Example 2 (\(a=-0.5\), and \(x_0=1\)).}
\begin{center}
  \setlength{\tabcolsep}{3pt} 
  \renewcommand{\arraystretch}{1.1} 
  \small 
  \begin{tabular}{@{}c|c|c|c|c|c|c|c|c@{}}
    \Xhline{1pt}
    \multirow{2}{*}{$N$} & \multicolumn{2}{c|}{$\alpha = \frac{1}{4}$} & \multicolumn{2}{c|}{$\alpha = \frac{1}{2}$} & \multicolumn{2}{c|}{$\alpha = \frac{3}{4}$} & \multicolumn{2}{c@{}}{$\alpha = 1$} \\ \cline{2-9}
    & $|Y_0 - Y^0|$ & $|Z_0 - Z^0|$ & $|Y_0 - Y^0|$ & $|Z_0 - Z^0|$ & $|Y_0 - Y^0|$ & $|Z_0 - Z^0|$ & $|Y_0 - Y^0|$ & $|Z_0 - Z^0|$ \\ \Xhline{1pt}
    8  & 1.3590E-04 & 2.4418E-05 & 1.2017E-04 & 1.0366E-04 & 1.0258E-04 & 1.9060E-04 & 8.2793E-05 & 2.8535E-04 \\
    16 & 3.4883E-05 & 4.9075E-06 & 3.0985E-05 & 2.6713E-05 & 2.6795E-05 & 4.9519E-05 & 2.2349E-05 & 7.3383E-05 \\
    32 & 8.8556E-06 & 1.1198E-06 & 7.8777E-06 & 6.8080E-06 & 6.8634E-06 & 1.2626E-05 & 5.8162E-06 & 1.8579E-05 \\
    64 & 2.2129E-06 & 2.3651E-07 & 1.9679E-06 & 1.6873E-06 & 1.7184E-06 & 3.1546E-06 & 1.4648E-06 & 4.6389E-06 \\
   128 & 5.4779E-07 & 5.0245E-08 & 4.8655E-07 & 4.1652E-07 & 4.2471E-07 & 7.8485E-07 & 3.6235E-07 & 1.1553E-06 \\
    \Xhline{1pt}
    CR    & 1.9888  & 2.2225  & 1.9873  & 1.9903  & 1.9795  & 1.9820  & 1.9603  & 1.9880  \\
    \Xhline{1pt}
  \end{tabular}
\end{center}
\label{Err-Y-Z-ex2-1}
\end{table*}

\begin{table}[!ht]
\caption{Run time({s}) of the scheme for Example 2 (\(a=-0.5\), and \(x_0=1\)).}
\centering
\setlength{\tabcolsep}{4pt} 
\renewcommand{\arraystretch}{1.1} 
\small 
\begin{tabular}{c|ccccc}
\toprule
\diagbox[width=6em,height=2.5em]{$\alpha$}{N} & 8 & 16 & 32 & 64 & 128 \\
\midrule
$1/4$ & 4.4760  & 8.9910  & 17.5931  & 35.7508  & 72.3901  \\
$1/2$ & 4.4913  & 8.8813  & 17.8321  & 35.4145  & 72.4596  \\
$3/4$ & 4.4974  & 8.9137  & 17.7552  & 35.3493  & 73.3658  \\
$1$   & 4.5146  & 8.8281  & 17.9691  & 36.1171  & 71.7899  \\
\bottomrule
\end{tabular}
\label{Err-Y-Z-ex2-1.5}
\end{table}
\vspace{-1.5em}
\begin{figure}[!ht]
\begin{center}
\includegraphics[
trim=0.000000in 0.000000in 0.000000in -0.398623in, 
width=0.9\textwidth] 
{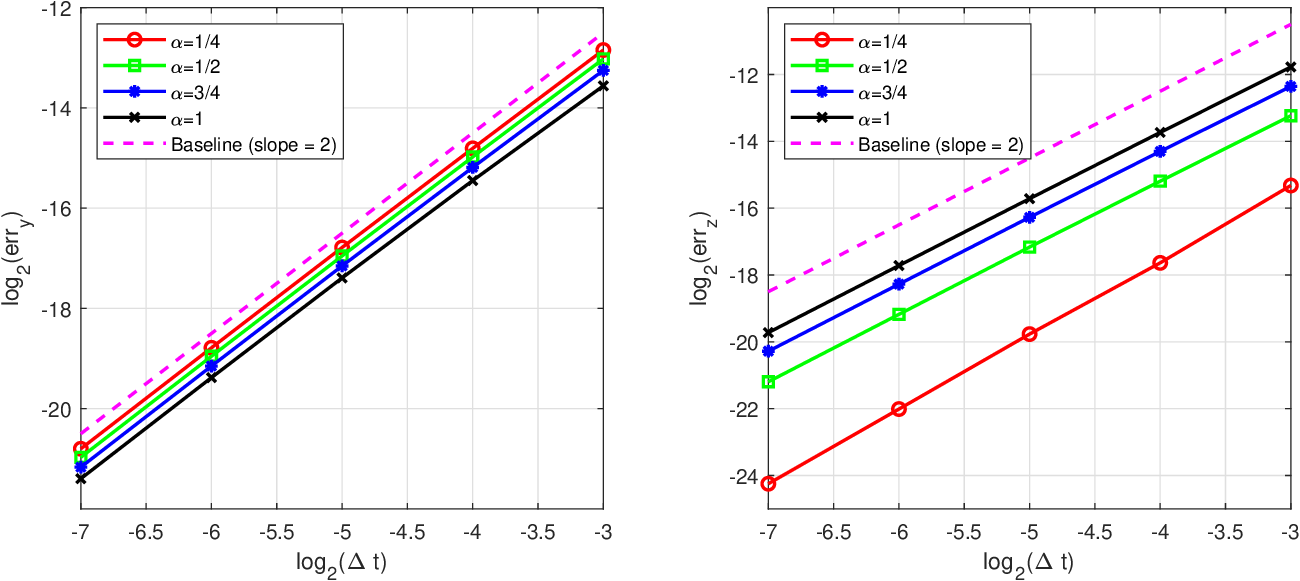} 
\caption{Log-log plot of $\text{err}_{y}$ against time interval \(\Delta t\) in Example 2 ($a=-0.5$, and $x_0=1$).}
\label{fig:Example2-0.5}
\end{center}
\end{figure}

\begin{table*}[!ht]
\caption{Error values and convergence rates for the scheme in Example 2 (\(a=-1\), and \(x_0=1.5\)).}
\begin{center}
  \setlength{\tabcolsep}{3pt} 
  \renewcommand{\arraystretch}{1.1} 
  \small 
  \begin{tabular}{@{}c|c|c|c|c|c|c|c|c@{}}
    \Xhline{1pt}
    \multirow{2}{*}{$N$} & \multicolumn{2}{c|}{$\alpha = \frac{1}{4}$} & \multicolumn{2}{c|}{$\alpha = \frac{1}{2}$} & \multicolumn{2}{c|}{$\alpha = \frac{3}{4}$} & \multicolumn{2}{c@{}}{$\alpha = 1$} \\ \cline{2-9}
    & $|Y_0 - Y^0|$ & $|Z_0 - Z^0|$ & $|Y_0 - Y^0|$ & $|Z_0 - Z^0|$ & $|Y_0 - Y^0|$ & $|Z_0 - Z^0|$ & $|Y_0 - Y^0|$ & $|Z_0 - Z^0|$ \\ \Xhline{1pt}

    8  & 2.6130E-04 & 1.0356E-04 & 3.1917E-04 & 3.7300E-04 & 3.7406E-04 & 6.6506E-04 & 4.2542E-04 & 9.8084E-04 \\
    16 & 6.6635E-05 & 2.2732E-05 & 8.2900E-05 & 9.6144E-05 & 9.8711E-05 & 1.7250E-04 & 1.1412E-04 & 2.5192E-04 \\
    32 & 1.6869E-05 & 5.3519E-06 & 2.1155E-05 & 2.4429E-05 & 2.5385E-05 & 4.3883E-05 & 2.9564E-05 & 6.3724E-05 \\
    64 & 4.2222E-06 & 1.2762E-06 & 5.3214E-06 & 6.1340E-06 & 6.4137E-06 & 1.1040E-05 & 7.4996E-06 & 1.5993E-05 \\
   128 & 1.0505E-06 & 3.0869E-07 & 1.3289E-06 & 1.5341E-06 & 1.6063E-06 & 2.7656E-06 & 1.8830E-06 & 4.0030E-06 \\
    \Xhline{1pt}
    CR  & 1.9897  & 2.0935  & 1.9777  & 1.9822  & 1.9671  & 1.9785  & 1.9567  & 1.9851  \\
    \Xhline{1pt}
  \end{tabular}
\end{center}
\label{Err-Y-Z-ex2-2}
\end{table*}
\vspace{-1.5em}

\begin{table}[!ht]
\caption{Run time({s}) of the scheme for Example 2 (\(a=-1\), and \(x_0=1.5\)).}
\centering
\setlength{\tabcolsep}{4pt} 
\renewcommand{\arraystretch}{1.1} 
\small 
\begin{tabular}{c|ccccc}
\toprule
\diagbox[width=6em,height=2.5em]{$\alpha$}{N} & 8 & 16 & 32 & 64 & 128 \\
\midrule
$1/4$ & 4.3599  & 8.4507  & 17.0215  & 34.1154  & 69.5319  \\
$1/2$ & 4.2851  & 8.5262  & 16.9074  & 33.9132  & 69.1714  \\
$3/4$ & 4.2347  & 8.3933  & 16.5688  & 33.4880  & 69.3673  \\
$1$   & 4.2455  & 8.3685  & 16.9407  & 33.5637  & 68.5035  \\
\bottomrule
\end{tabular}
\label{Err-Y-Z-ex2-2.5}
\end{table}
\clearpage
\begin{figure}[!ht]
\begin{center}
\includegraphics[
trim=0.0in 0.0in 0.0in -0.4in, 
width=0.9\textwidth] 
{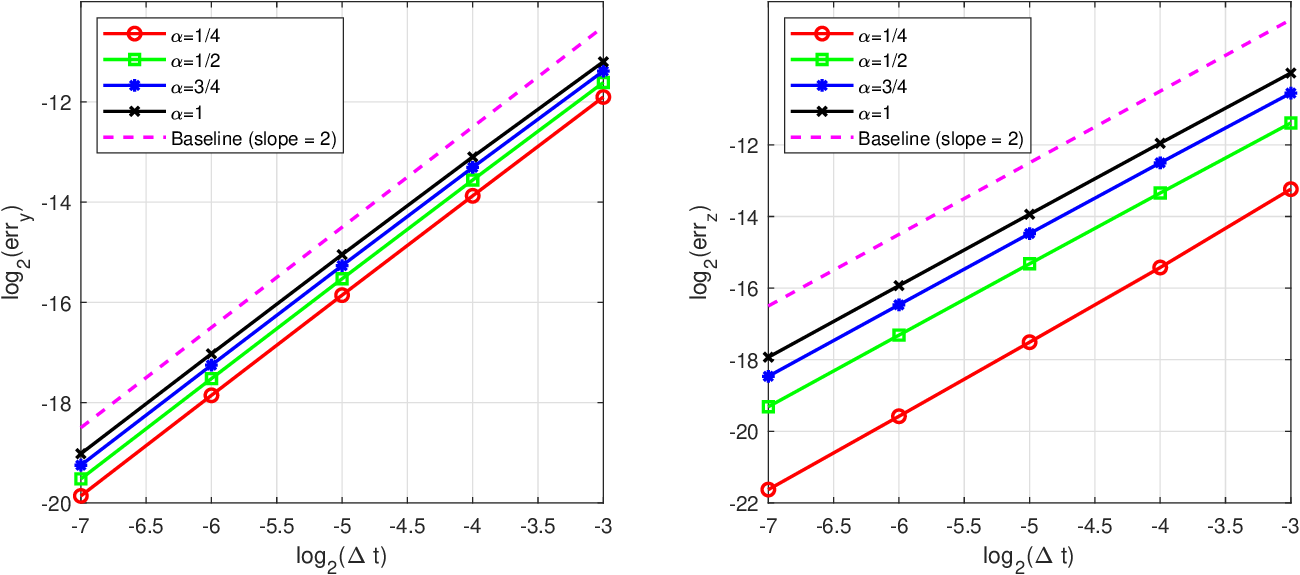}
\caption{Log-log plot of $\text{err}_{y}$ against time interval \(\Delta t\) in Example 2 ($a=-1$, and $x_0=1.5$).}
\label{fig:Example2-2}
\end{center}
\end{figure}

\vspace{1.0em}
\section{Conclusions}

We have proposed a novel explicit one step scheme for solving decoupled forward backward stochastic differential equations. The stability analysis of the proposed numerical scheme has been conducted. We subsequently provided a rigorous proof that the proposed scheme achieves the second order convergence rates. Additionally, when $\alpha = 1$, our scheme degenerates into the Crank-Nicolson method as presented in \cite{CDMK14,ZWCLPS06}, making it a special case within our framework. Furthermore, this work provides a reference for how we can develop a new one step higher order scheme.

\vskip 3mm

\textbf{ Declarations}

The authors confirm that there are no conflicts of interest regarding this publication. Since this study did not involve the generation or analysis of any datasets, data sharing is not applicable. All authors made equal contributions to the research and have reviewed and approved the final version of the manuscript.

\textbf{ Acknowledgements}

The authors extend their gratitude to the editor for overseeing the review process and to the referees for their insightful feedback and recommendations.

\end{document}